\documentclass[fullpage]{amsart}

\usepackage{amsthm,amssymb,amsmath,a4wide,graphicx,tikz,bm}
\usepackage[english]{babel}
\usepackage{subfigure}
\usepackage{pgfplots}
\usepackage{mathtools}
\usepackage[hidelinks]{hyperref}

\usetikzlibrary{patterns,calc,positioning,shapes}
\pgfplotsset{compat=1.10}
\usepgfplotslibrary{fillbetween}

\title{Random Lochs' Theorem}
\author[Charlene Kalle]{Charlene Kalle}
\address[Charlene Kalle]{Mathematisch Instituut, Leiden University, Niels Bohrweg 1, 2333CA Leiden, The Netherlands}
\email[Charlene Kalle]{kallecccj@math.leidenuniv.nl}
\author[Evgeny Verbitskiy]{Evgeny Verbitskiy}
\address[Evgeny Verbitskiy]{Mathematisch Instituut, Leiden University, Niels Bohrweg 1, 2333CA Leiden, The Netherlands \textit{and} Bernoulli Institute, Groningen University, The Netherlands}
\email[Evgeny Verbitskiy]{evgeny@math.leidenuniv.nl}
\author[Benthen Zeegers]{Benthen Zeegers$^\dagger$}
\address[Benthen Zeegers]{Mathematisch Instituut, Leiden University, Niels Bohrweg 1, 2333CA Leiden, The Netherlands}
\email[Benthen Zeegers]{b.p.zeegers@math.leidenuniv.nl}

\begin{document}
\thanks{$^\dagger$ The third author would like to thank the University of Vienna for their hospitality while working on part of this project.}

\begin{abstract}
In 1964 Lochs proved a theorem on the number of continued fraction digits of a real number $x$ that can be determined from just knowing its first $n$ decimal digits. In 2001 this result was generalised to a dynamical systems setting by Dajani and Fieldsteel, where it compares sizes of cylinder sets for different transformations. In this article we prove a version of Lochs' Theorem for random dynamical systems as well as a corresponding Central Limit Theorem. The main ingredient for the proof is an estimate on the asymptotic size of the cylinder sets of the random system in terms of the fiber entropy. To compute this entropy we provide a random version of Rokhlin's formula for entropy.
\end{abstract}

\maketitle

\newtheorem{prop}{Proposition}[section]
\newtheorem{theorem}{Theorem}[section]
\newtheorem{lemma}{Lemma}[section]
\newtheorem{cor}{Corollary}[section]
\newtheorem{remark}{Remark}[section]
\theoremstyle{definition}
\newtheorem{defn}{Definition}[section]
\newtheorem{ex}{Example}[section]

\maketitle

\section{Introduction}
Real numbers can be represented in many different ways, e.g.~by binary, decimal or continued fraction expansions, and one can wonder about the amount of information that each one of these expansions carries. In 1964 Lochs considered a specific question of this form: Given the first $n$ decimal digits of a further unknown irrational number $x \in (0,1)$, what is the largest number $m=m(n,x)$ of regular continued fraction digits of $x$ is that can be determined from this information. Lochs \cite{lochs} answered this question for the limit $n \rightarrow \infty$ by showing that for Lebesgue almost every $x \in (0,1)$,
\begin{align}\label{eq:1.1}
\lim_{n \rightarrow \infty} \frac{m(n,x)}{n} = \frac{6 \log 2 \log 10}{\pi^2}.
\end{align}
Over the years Lochs' result has been refined and generalised in many directions. Let $\lambda$ denote the one-dimensional Lebesgue measure. In \cite{faivre-clt} Faivre established a Central Limit Theorem associated to Lochs' Theorem:
\begin{align}\label{q:faivre}
\lim_{n \rightarrow \infty} \lambda\bigg( \Big\{ x \in (0,1): \frac{m(n,x) - n \frac{6 \log 2 \log 10}{\pi^2}}{c_1 \sqrt{n}} \leq u \Big\} \bigg) = \frac{1}{\sqrt{2\pi}}\int_{-\infty}^u e^{-t^2 / 2} dt
\end{align}
holds for some constant $c_1 > 0$. See \cite{faivre-ldp,faivre01,wu06,wu08} for other results related to the limit in \eqref{eq:1.1} and \cite{li08,barreira08,fang16,fang} for results where the decimal expansions in \eqref{eq:1.1} are replaced by $\beta$-expansions.

\vskip .2cm
In \cite{bosma99} Bosma, Dajani and Kraaikamp highlighted that Lochs' Theorem can be seen as a dynamical statement. This viewpoint was further developed in \cite{dajani01}, where Dajani and Fieldsteel gave the dynamical equivalent of the local limit statement from \eqref{eq:1.1} for what they called {\em number theoretic fibered maps (NTFM)}. An NTFM is a triple $(T,\mu,\mathcal D)$ where $T: [0,1) \rightarrow [0,1)$ is a surjective map, $\mu$ is a Borel measure and $\mathcal D = \{ D_1, D_2, \ldots \}$ is an at most countable interval partition of $[0,1)$, such that
\begin{itemize}
\item[(n1)] $T|_D$ is continuous and strictly monotone on each element $D \in \mathcal D$;
\item[(n2)] $\mu$ is an ergodic invariant probability measure for $T$ that is equivalent to $\lambda$ with a density that is bounded and bounded away from zero;
\item[(n3)] the partition $\mathcal D$ generates the Borel $\sigma$-algebra $\mathcal B$ on $[0,1)$ in the sense that if for each $n$ we use
\begin{equation}\label{q:cylindersbar}
\mathcal D_n = \bigvee_{k=0}^{n-1} T^{-k} \mathcal D = \{ D_1 \cap T^{-1} D_2 \cap \cdots \cap T^{-(n-1)} D_n \, : \, D_k \in \mathcal D, \, 1 \le k \le n \}
\end{equation}
to denote the {\em level $n$ cylinders} of $T$, then the smallest $\sigma$-algebra containing all these sets for all $n \ge 1$, denoted by $\sigma \big( \bigvee_{k\ge 0} T^{-k} \mathcal D)$, equals $\mathcal B$ up to sets of Lebesgue measure zero;
\item[(n4)] the entropy $-\sum_{D \in \mathcal D}\mu(D)\log \mu(D)$ of $\mathcal D$ with respect to $\mu$ is finite.
\end{itemize}
The name NTFM refers to the fact that an NTFM generates for each $x \in [0,1)$ a digit sequence $(d^T_k (x))_{k \ge 1}$ by setting $d_k^T(x) = j$ if $T^{k-1}(x) \in D_j$, and for certain maps $T$ these sequences correspond to well-known number expansions. Number expansions that can be obtained in this way include $\beta$-expansions, various L\"uroth-type expansions and various types of continued fraction expansions. 

\vskip .2cm
If for two NTFM's $(T, \mu, \mathcal D)$ and $(S, \tilde \mu, \mathcal E)$ we define the number
\[ m_{T,S} (n,x) = \sup\{m \in \mathbb{N}: \mathcal D_n(x) \subseteq \mathcal E_m(x)\},\]
where $\mathcal D_n (x)$ and $\mathcal E_m(x)$ denote the elements of the partitions $\mathcal D_n$ and $\mathcal E_m$ as in \eqref{q:cylindersbar} that contain $x$, respectively, then one can interpret $m_{T,S} (n,x)$ as the largest $m$ so that the level $m$ cylinder for $S$ containing $x$ can be determined from knowing only the level $n$ cylinder for $T$ that contains $x$. Equivalently, if we use $(d_k^T(x))_{k \ge 1}$ and $(d_k^S(x))_{k \ge 1}$ to denote the digit sequences produced by $T$ and $S$, $m_{T,S} (n,x)$ is the largest $m$ such that the digits $d_1^S (x), \ldots , d_m^S(x)$ can be determined from knowing $d_1^T (x), \ldots , d_n^T(x)$ of a further unknown $x \in [0,1)$. The authors of \cite{dajani01} proved that for any two NTFM's $(T, \mu, \mathcal D)$ and $(S, \tilde \mu, \mathcal E)$ with measure theoretic entropies $ h_\mu(T), h_{\tilde \mu} (S) > 0$ it holds that
\begin{equation}\label{q:df}
\lim_{n\rightarrow \infty} \frac{m_{T,S}(n,x)}{n} = \frac{h_\mu(T)}{h_{\tilde \mu}(S)} \qquad \lambda\text{-a.e.}
\end{equation}
Lochs' original result can be recovered by taking $T (x) = 10 x \pmod 1$ and $S (x) = \frac1x \pmod 1$ (the Gauss map). Similar to the result in \eqref{q:faivre} by Faivre, Herczegh \cite{atilla} proved a Central Limit Theorem for the statement in \eqref{q:df} for a specific class of pairs of NTFM's. For two NTFM's $(T, \mu, \mathcal D)$ and $(S, \tilde \mu, \mathcal E)$ that satisfy $ h_\mu(T), h_{\tilde\mu} (S) > 0$ and several additional conditions, he proved in \cite[Corollary 2.1]{atilla} that for each $u \in \mathbb R$,
\begin{equation}\label{q:clther}
\lim_{n \rightarrow \infty} \tilde \mu\bigg( \Big\{ x \in (0,1): \frac{m_{T,S}(n,x)-n \frac{h_\mu(T)}{h_{\tilde \mu}(S)}}{c_2 \sqrt n} \le u \Big\} \bigg) = \frac{1}{\sqrt{2\pi}}\int_{-\infty}^u e^{-t^2 / 2} \, dt
\end{equation}
for some appropriate constant $c_2>0$.

\vskip .2cm
The past decades have seen an increasing interest in random dynamical systems that generate number expansions. This is partly due to the fact that often such systems generate for a typical $x \in [0,1)$ not just one, but uncountably many different number expansions of a given type. Instead of using a single transformation, a {\em random dynamical system} consists of a family of maps $\{T_i: X \to X\}_{i \in I}$, for some index set $I$, each acting on the same state space $X$ and from which one map is chosen at each time step according to some probability law. There are various random dynamical systems related to number expansions. The random $\beta$-transformation was first introduced in \cite{DK03} and then further investigated in \cite{DV05,invariant,DK07,DK13,Kem14,BD17,DJ17,Suz19,DKM21}. Interesting features of this system are its relation to Bernoulli convolutions, see e.g.~\cite{JSS11,DK13,Kem14}, and to $\beta$-encoders, see e.g~\cite{DDGV02,DGWY,Gun,KHTA,JMKA,Makino,SJO,JM} and the references therein. A random system producing binary expansions was studied in \cite{DK20}, random dynamical systems related to continued fraction expansions appear in \cite{KKV17,DO18,BRS20,AFGV,DKM21,KMTV} and random L\"uroth maps are considered in \cite{KM21,KM}.

\vskip .2cm
In this article we extend the version of Lochs' Theorem in \eqref{q:df} to the setting of random dynamical systems and give a corresponding Central Limit Theorem in the spirit of \eqref{q:clther}. The class of random dynamical systems we consider, which we call {\em random number systems}, contains the class of deterministic NTFM's and all of the random dynamical systems related to number expansions mentioned above. A random number system consists of a family of maps $ \{ T_i : [0,1) \to [0,1) \}_{i \in I}$, where the index set $I$ is a possibly uncountable Polish space, each map $T_i:[0,1) \to [0,1)$, $i \in I$, admits an appropriate partition $\alpha_i = \{A_{i,0}, A_{i,1}, \ldots\}$ of $[0,1)$ and there exists an appropriate probability measure $\mu$ on $I^\mathbb N \times [0,1)$. (The precise definition will be given in the next section.) Thus a random number system is a quintuple $\mathcal T = (I, \mathbf P,  \{ T_i  \}_{i \in I}, \mu, \{ \alpha_i\}_{i \in I})$, where $\mathbf P$ is the probability law on $I^\mathbb N$ determining the random choices.

\vskip .2cm
Each $\omega \in I^\mathbb N$ gives for each $x \in [0,1)$ a {\em random orbit} by setting $T_{\omega}^0(x) =x$ and, for each $n \ge 1$,
\[ T_{\omega}^n(x) = T_{\omega_1 \cdots \omega_n}(x)  =
T_{\omega_n} \circ T_{\omega_{n-1}} \circ \cdots \circ T_{\omega_1}(x). \]
The partitions $\alpha_i$, $i \in I$, lead for each $(\omega,x)$ to a digit sequence $(d^\mathcal T_n(\omega,x))_{n \ge 1}$ by setting
\begin{equation}\label{q:digitsequence}
d_n^\mathcal T(\omega,x) = j \quad \text{if } \, T_\omega^{n-1}(x) \in A_{\omega_n,j}.
\end{equation}
Analogous to \eqref{q:cylindersbar}, for a random number system $\mathcal T$ the {\em random level $n$ cylinders} are defined as follows: For each $\omega \in I^{\mathbb{N}}$ and $n \in \mathbb{N}$ we define the partition
\begin{equation}\begin{split}\label{q:cylindersr}
\alpha_{\omega,n} =\ & \bigvee_{k = 0}^{n-1} (T_\omega^k)^{-1} \alpha_{\omega_{k+1}} \\
=\ & \{ A_{\omega_1,j_1}\cap T_{\omega_1}^{-1}A_{\omega_2,j_2} \cap \cdots \cap T_{\omega_1 \cdots \omega_{n-1}}^{-1} A_{\omega_n,j_n} \, : \, A_{\omega_k,j_k} \in \alpha_{\omega_k}, \, 1\le k \le n\}.
\end{split}\end{equation}
Write $\alpha_{\omega,n}(x)$ for the random $(n, \omega)$-cylinder that contains $x$. Given two random number systems $\mathcal{T}=(I, \mathbf P, \{ T_i \}_{i \in I},\mu, \{ \alpha_i\}_{i \in I})$ and $\mathcal S = (J, \mathbf Q, \{S_j\}_{j \in J}, \rho, \{\gamma_j \}_{j \in J})$, for each $n \in \mathbb{N}$, $\omega \in I^{\mathbb N}$, $\tilde{\omega} \in J^{\mathbb N}$ and $x \in [0,1)$, let
\begin{align}\label{eq:1.15}
m_{\mathcal T,\mathcal S}(n,\omega,\tilde{\omega},x) = \sup\{m \in \mathbb{N}: \alpha_{\omega,n}(x) \subseteq \gamma_{\tilde{\omega},m}(x)\}.
\end{align}
This quantity can be interpreted as follows: For given $\omega \in I^{\mathbb N}$ and $\tilde{\omega} \in J^{\mathbb N}$, $m_{\mathcal T,\mathcal S}(n,\omega,\tilde{\omega},x)$ is the largest level $m$ for which we can determine the random $(m, \tilde \omega)$-cylinder for $\mathcal S$ containing $x$ from knowing only the random $(n, \omega)$-cylinder for $\mathcal T$ that containing $x$. Alternatively, it is the largest $m$ such that $d_1^\mathcal S(\tilde{\omega},x),\ldots, d_m^\mathcal S(\tilde{\omega},x)$ can be determined from knowing the digits $d_1^\mathcal T(\omega,x),\ldots, d_n^\mathcal T(\omega,x)$ of a further unknown $x \in [0,1)$. In this article we obtain the following Random Lochs' Theorem, where the measure theoretic entropy from \eqref{q:df} is replaced by \emph{fiber entropy}, which for a random number system $\mathcal{T}$ is a quantity $h^{\text{fib}}(\mathcal{T}) \in [0,\infty)$ and will be defined in the next section.

\begin{theorem}\label{t:randomlochs}
Let $\mathcal{T} = (I, \mathbf P,  \{ T_i  \}_{i \in I}, \mu, \{ \alpha_i\}_{i \in I})$ and $\mathcal{S} = (J, \mathbf Q, \{S_j\}_{j \in J}, \rho, \{\gamma_j \}_{j \in J})$ be two random number systems. If $h^{\mathrm{fib}}(\mathcal{T}),h^{\mathrm{fib}}(\mathcal{S}) > 0$, then
\[ \lim_{n \rightarrow \infty} \frac{m_{\mathcal T,\mathcal S}(n,\omega,\tilde{\omega},x)}{n} = \frac{h^{\mathrm{fib}}(\mathcal{T})}{h^{\mathrm{fib}}(\mathcal{S})} \qquad \lambda\text{-a.e.}\]
for $\mathbf P \times \mathbf Q$-a.a.~$(\omega,\tilde{\omega}) \in I^{\mathbb{N}} \times J^{\mathbb{N}}$.
\end{theorem}

\vskip .2cm
We like to make two remarks about this result. Firstly, the quotient of measure theoretic entropies that appears as the value of the limit in the deterministic setting has been replaced by a quotient of fibered entropies in the random setting. Secondly, the setup allows for the index set $I$ of the family $\{ T_i \}_{i \in I}$ to be uncountable, so that the results apply to e.g.~random $\beta$-transformations where the value of $\beta$ can range over a whole interval, see Example~\ref{ex2.3} below. This makes the proofs more involved. For an example of a proof of Theorem~\ref{t:randomlochs} in case $I$ is countable (and under some additional assumptions), we refer to \cite[Proposition 6.28]{zeegers}.

\vskip .2cm
In \cite{dajani01} an essential ingredient to prove \eqref{q:df} is the following general result on interval partitions. If $\mathcal{P} = \{P_n\}_{n=1}^{\infty}$ is a sequence of interval partitions and $c \geq 0$, we say that \emph{$\mathcal{P}$ has entropy $c$ $\lambda$-a.e.}~if
\[ \lim_{n \rightarrow \infty} - \frac{\log \lambda(P_n(x))}{n} = c \qquad \lambda\text{-a.e.},\]
where $P_n(x)$ denotes the element of the partition $P_n$ containing $x$. 
\begin{theorem}[Theorem 4 of \cite{dajani01}]\label{t:partitions}
 Let $\mathcal{P} = \{P_n\}_{n=1}^{\infty}$ and $\mathcal{Q} = \{Q_n\}_{n=1}^{\infty}$ be two sequences of interval partitions.  For each $n \in \mathbb{N}$ and $x \in [0,1)$, put
\[ m_{\mathcal{P},\mathcal{Q}}(n,x) = \sup\{ m \in \mathbb{N} : P_n(x) \subseteq Q_m(x) \}.\]
Suppose that $\mathcal{P}$ has entropy $c \in (0,\infty)$ $\lambda$-a.e.~and $\mathcal{Q}$ has entropy $d \in (0,\infty)$ $\lambda$-a.e. Then
\[ \lim_{n \rightarrow \infty} \frac{m_{\mathcal{P},\mathcal{Q}}(n,x)}{n} = \frac{c}{d} \qquad \lambda\text{-a.e.}\]
\end{theorem}

The proof of \eqref{q:df} goes roughly along the following lines. An application of the Kolmogorov-Sinai Theorem and of the Shannon-McMillan-Breiman Theorem to the NTFM's $T$ and $S$ provides the appropriate asymptotics for the size of the cylinder sets from \eqref{q:cylindersbar} for both maps $T$ and $S$ to establish the positive entropy conditions and then Theorem~\ref{t:partitions} completes the proof. To achieve Theorem~\ref{t:randomlochs} we also employ Theorem~\ref{t:partitions} and therefore the main achievement here is obtaining the right asymptotics for the size of the random cylinder sets from \eqref{q:cylindersr}. More precisely Theorem~\ref{t:randomlochs} will appear as a corollary of the following theorem.

\begin{theorem}\label{t:main}
Let $\mathcal{T} = (I, \mathbf P, \{T_i\}_{i \in I},\mu, \{\alpha_i\}_{i \in I})$ be a random number system. The following hold,
\begin{itemize}
\item[(i)] For $\mathbf P$-a.a.~$\omega \in I^{\mathbb{N}}$ we have
\[ \lim_{n \rightarrow \infty} - \frac{\log \lambda\big(\alpha_{\omega,n}(x)\big)}{n} = h^{\mathrm{fib}}(\mathcal{T}), \qquad \lambda\text{-a.e.}\]
\item[(ii)] Let $\nu$ denote the marginal of $\mu$ on $I^{\mathbb N}$. Furthermore, let $F$ be the skew product on $I^{\mathbb N} \times [0,1)$ given by $F(\omega,x) = (\tau \omega, T_{\omega_1}(x))$, where $\tau$ denotes the left shift on $I^{\mathbb N}$. If $h_\nu(\tau) < \infty$, then
\[ h^{\mathrm{fib}}(\mathcal{T}) = h_\mu(F) - h_\nu (\tau).\]
\item[(iii)] If for each $i \in I$ and $A \in \alpha_i$  the restriction $T_i|_A$ is differentiable, then
\[ h^{\mathrm{fib}}(\mathcal{T}) = \int_{I^\mathbb N \times [0,1)} \log |DT_{\omega_1}(x)| \, d\mu(\omega,x).\]
\end{itemize}
\end{theorem}
The first part of this theorem gives the required estimates for the asymptotic sizes of the cylinder sets from \eqref{q:cylindersr} and, when combined with Theorem~\ref{t:partitions}, leads to Theorem~\ref{t:randomlochs}. The limit from Theorem~\ref{t:randomlochs} is expressed in terms of the fiber entropies of the two random number systems. Parts (ii) and (iii) of Theorem~\ref{t:main} give different ways to determine this limit. The second part works in case the entropy of the marginal of $\mu$ on $I^\mathbb N$ is finite. The third part gives a random version of Rokhlin's formula for entropy.

\vskip .2cm
We also prove a Central Limit Theorem for Theorem~\ref{t:randomlochs} in case we compare the digits obtained from a random number system $\mathcal T = (I, \mathbf P, \{T_i\}_{i \in I}, \mu, \{ \alpha_i\}_{i \in I})$ to those from an NTFM $(S, \tilde\mu, \mathcal E)$ under additional assumptions on both systems. To be more specific, for such systems we obtain that for all $u \in \mathbb R$,
\begin{equation}\label{q:randomclt}
\lim_{n \rightarrow \infty} \mu\Big(\Big\{(\omega,x) \in I^\mathbb N \times [0,1): \frac{m_{\mathcal T,S}(n,\omega,x) - n\frac{h^{\text{fib}}(\mathcal T)}{h_{\tilde \mu}(S)}}{\kappa \sqrt n} \leq u\Big\}\Big) = \frac{1}{\sqrt{2 \pi}} \int_{-\infty}^u e^{-t^2/2} dt,
\end{equation}
for an appropriate constant $\kappa>0$.

\vskip .2cm

\vskip .2cm
The article is outlined as follows. In Section~\ref{s:prel} we give a precise definition of random number systems and fiber entropy and we provide some preliminaries. In Section~\ref{s:main} we prove Theorem~\ref{t:main} and obtain Theorem~\ref{t:randomlochs} as a corollary and we prove the Central Limit Theorem from \eqref{q:randomclt}. In Section~\ref{s:examples} we provide some examples.

\section{Preliminaries}\label{s:prel}

\subsection{Partitions and entropy}
We first quickly recall the definition of measure theoretic entropy, starting with partitions. Let $(X, \mathcal F,m)$ be an arbitrary measure space. We call a collection $\mathcal P$ a {\em partition} of $(X, \mathcal F,m)$ if it is an at most countable collection of measurable sets, $\mathcal P \subseteq \mathcal F$, that are pairwise disjoint and satisfy $X = \bigcup_{P \in \mathcal P} P$, where both properties are considered modulo $m$-null sets. If $\mathcal P$ is a partition of $X$ and $x \in X$, then we denote by $\mathcal P(x)$ the partition element of $\mathcal P$ containing $x$. For two partitions $\mathcal P_1$ and $\mathcal P_2$ of $(X, \mathcal F,m)$ we use the notation $\mathcal P_1 \le \mathcal P_2$ to indicate that $\mathcal P_2$ is a {\em refinement} of $\mathcal P_1$, i.e., that for every $P \in \mathcal P_2$ there is a $Q \in \mathcal P_1$ such that $P \subseteq Q$. Moreover, we use $\mathcal P_1 \vee \mathcal P_2 := \{ P \cap Q \, : \, P \in \mathcal P_1, \, Q \in \mathcal P_2 \}$ to denote the {\em common refinement} of $\mathcal P_1$ and $\mathcal P_2$.

\vskip .2cm
The {\em entropy} of a partition $\mathcal P$ with respect to the measure $m$ is defined as
\[ H_m(\mathcal P) = - \sum_{P \in \mathcal P} m(P) \log m(P).\]
If $T: X \to X$ is measurable and measure preserving with respect to $m$, then the {\em entropy of $\mathcal P$ with respect to the transformation $T$} is defined as
\[ h_m (\mathcal P, T) := \lim_{n \to \infty} \frac1n H_m \Big( \bigvee_{k=0}^{n-1} T^{-k} \mathcal P \Big)\]
and the {\em measure theoretic entropy} of $T$ with respect to $m$ is given by
\[ h_m(T) := \sup_{\mathcal P} h_m(\mathcal P, T),\]
where the supremum is taken over all partitions $\mathcal P$ satisfying $H_m (\mathcal P) < \infty$. For any collection of measurable sets $\mathcal E \subseteq \mathcal F$ we use $\sigma(\mathcal E)$ to denote the smallest $\sigma$-algebra containing $\mathcal E$. We say that $\mathcal E$ {\em generates} $\mathcal F$ if $\sigma(\mathcal E)= \mathcal F$ up to sets of $m$-measure zero. The Kolmogorov-Sinai Theorem facilitates the computation of measure theoretic entropy, since it says that $h_m(T) = h_m(\mathcal P,T)$ for any partition $\mathcal P$ of finite entropy that generates the $\sigma$-algebra $\mathcal F$. More generally, the following result, which can be found in e.g.~\cite[Proposition 9.3.1]{kalle2}, holds.

\begin{lemma}\label{l:increasingentropy}
If $\mathcal P_1 \le \mathcal P_2 \le \ldots$ is an increasing sequence of finite partitions of $(X, \mathcal F,m)$ such that $\sigma(\lim_{n \rightarrow \infty} \mathcal P_n) = \mathcal F$ up to sets of $m$-measure zero, then $h_m(T) = \lim_{n \to \infty} h_m(\mathcal P_n, T)$.
\end{lemma}

\subsection{Random number systems}\label{subsec2.2}
We now define the dynamical systems that we are interested in. Let $I$ be a Polish space with Borel $\sigma$-algebra $\mathcal B_I$. Let $\tau$ denote the left shift on $I^\mathbb N$ and $\mathbf P$ a Borel probability measure on the product $\sigma$-algebra $\mathcal B_I^\mathbb N$ such that $\tau$ is non-singular with respect to $\mathbf{P}$, i.e.~$\mathbf{P}(\tau^{-1} A) = 0$ if $\mathbf{P}(A) = 0$ for all $A \in \mathcal B_I^{\mathbb{N}}$. We call $(I^\mathbb N, \mathcal B_I^\mathbb N, \mathbf P)$ the {\em base space}. For $i_1, \ldots, i_n \in I$, $n \ge 1$, we let
\[ [i_1, i_2, \ldots, i_n] = \{ \omega \in I^\mathbb N \, : \, \omega_j = i_j, \, 1 \le j \le n \}.\]
For each $i \in I$, let $T_i : [0,1) \to [0,1)$ be a Borel measurable transformation. Let $\mathcal B$ denote the Borel $\sigma$-algebra on $[0,1)$ and $\lambda$ the Lebesgue measure on $[0,1)$. Associated to the family $\{ T_i :[0,1) \to [0,1) \}_{i \in I}$ is the {\em skew product transformation} or {\em random map}
\[ F: I^\mathbb N \times [0,1) \to  I^\mathbb N \times [0,1), \, (\omega,x) \mapsto \big( \tau \omega, T_{\omega_1}(x) \big).\]
Let $\mu$ be an invariant probability measure for $F$ on $I^{\mathbb{N}} \times [0,1)$. For each $i \in I$ let $\alpha_i = \{ A_{i,0}, A_{i,1}, \ldots\}$ be a partition of $[0,1)$ by countably many subintervals of $[0,1)$, possibly containing empty sets. (For notational convenience we add countably many empty sets to $\alpha_i$ in case it would naturally be a finite set.) A {\em random number system} is a collection $\mathcal{T}=(I, \mathbf P, \{T_i\}_{i \in I},\mu, \{\alpha_i\}_{i \in I})$ on $[0,1)$ that satisfies the following conditions.
\begin{itemize}
\item[(r1)] The map $I \times [0,1) \ni (i,x) \mapsto T_i (x) \in [0,1)$ is measurable.
\item[(r2)] For each $i \in I$ and $A \in \alpha_i$, $T_i|_{A}$ is strictly monotone and continuous.
\item[(r3)] The partition $\Delta = \{ \Delta(j) \, : \, j \ge 0 \}$ of $I^\mathbb N \times [0,1)$ given by \begin{equation}\label{q:delta}
\Delta(j) = \{ (\omega,x) \, : \, x \in A_{\omega_1,j} \} = \bigcup_{i \in I} [i] \times A_{i,j} \quad \text{ for each } j \ge 0
\end{equation}
is measurable, i.e., $\Delta(j) \in \mathcal B_I^\mathbb N \times \mathcal B$ for all $j \ge 0$.
\item[(r4)] For $\mathbf P$-a.a.~$\omega \in I^{\mathbb{N}}$ we have that, for all $B \in \mathcal{B}$, $\lambda\big(T_{\omega_1}^{-1}B\big) = 0$ if $\lambda(B) = 0$.
\item[(r5)] For $\mathbf P$-a.a.~$\omega \in I^{\mathbb{N}}$ we have $\sigma\big(\displaystyle \lim_{n \to \infty}\alpha_{\omega,n}\big) = \mathcal{B}$ up to sets of $\lambda$-measure zero.
\item[(r6)] The $F$-invariant measure $\mu$ is ergodic and equivalent to $\mathbf P \times \lambda$.
\item[(r7)] The entropy of $\Delta$ with respect to $\mu$ is finite.
\end{itemize}

Most of the conditions (r1)--(r7) are easily verified in specific applications and not very restrictive. We give some comments on them.
\vskip .1cm
- Conditions (r1) and (r3) are typical measurability conditions and are immediate in case $I$ is at most countable (and equipped with the discrete topology). It easily follows from (r1) that the random map $F$ is measurable.
\vskip .1cm
- Condition (r2) is needed to get digit sequences $(d_n^\mathcal T (\omega,x))_{n \ge 1}$ as in \eqref{q:digitsequence}. It follows from (r5) that, for $\mathbf P$-a.a.~$\omega \in I^{\mathbb{N}}$, knowing $(d_n^{\mathcal T}(\omega,x))_{n \geq 1}$ determines $x \in [0,1)$ uniquely $\lambda$-a.e.
\vskip .1cm
- Condition (r4) is a form of fiberwise non-singularity and from (r6) it follows that $\mu$ is the only probability measure that is both $F$-invariant and absolutely continuous w.r.t.~$\mathbf P \times \lambda$. In case $I$ is countable, then it is easy to verify that (r4) already follows from only assuming (r6).
\vskip .1cm
- If we let $\pi_I: I^\mathbb N \times [0,1) \to I^\mathbb N$ be the canonical projection onto the first coordinate and write $\nu = \mu \circ \pi_I^{-1}$ for the marginal of the invariant measure $\mu$ on $I^\mathbb N$, then from (r6) it follows that $\nu$ is $\tau$-invariant, ergodic and equivalent to $\mathbf P$. In particular, if $\mathbf P$ is $\tau$-invariant, then $\nu=\mathbf P$.
\vskip .1cm
- Condition (r7) guarantees that the fiber entropy defined later on is well defined. Note that if $\Delta$ is a finite set, then (r7) is automatically satisfied.

\vskip .2cm

One of the consequences of (r1)--(r7) is that each random number system admits a {\em system of conditional measures}, i.e., a family of probability measures $\{\mu_{\omega}\}_{\omega \in I^\mathbb N}$ such that
\begin{itemize}
\item $\mu_{\omega}$ is a probability measure on $([0,1),\mathcal{B})$ for $\nu$-a.a.~$\omega \in I^{\mathbb N}$,
\item for any $f \in L^1(\mu)$ the map $I^{\mathbb N} \ni \omega \mapsto \int_{[0,1)} f(\omega,x) \, d\mu_{\omega}(x)$ is measurable and
\begin{equation}\label{q:system}
\int_{I^\mathbb N \times [0,1)} f \, d\mu = \int_{I^{\mathbb N}} \Big(\int_{[0,1)} f(\omega,x) \, d\mu_{\omega}(x)\Big) d\nu(\omega).
\end{equation}
\end{itemize}
Moreover, if $\{\tilde{\mu}_{\omega}\}_{\omega \in I^{\mathbb N}}$ is another system of conditional measures for $\mu$, then $\mu_{\omega} = \tilde{\mu}_{\omega}$ for $\nu$-a.a.~$\omega \in I^{\mathbb N}$. (See \cite[Theorem 1.0.8]{Aar97} together with \cite[Proposition 5.1.7]{foundations} for a justification.)

\vskip .2cm

We now present some classes of systems for which the assumptions from the definition of random number system are satisfied.
\vskip .1cm
- A class of maps that satisfy (r2) and (r4) and that are well studied in literature is the class of {\em Lasota-Yorke type maps}. A Lasota-Yorke type map is a map $T:[0,1) \to [0,1)$ that is piecewise monotone $C^2$ and non-singular with $|DT(x)| >0$ for all $x$ where the derivative is defined. In that case an obvious candidate for the partition from (r2) is the partition of $[0,1)$ given by the maximal intervals on which $T$ is monotone.
\vskip .1cm
- Given a family $\{ T_i\}_{i \in I}$ of Lasota-Yorke type maps for some appropriate index set $I$, a sufficient condition for $\Delta$ to be a generator in the sense of (r5) is that $\inf_{(i,x)} |DT_i(x)| > 1$. In case $I$ is finite, this is equivalent to the condition that each $T_i$ is expanding. We can allow for neutral fixed points as well and still get (r5) if we assume that the branches of the maps are full, i.e., map onto the whole interval $(0,1)$, and expanding outside each neighborhood of the neutral fixed point. Examples include the Gauss-R\'enyi map from \cite{KKV17} that we will encounter in Example \ref{x:gr} below and Manneville-Pomeau type maps as in e.g.~\cite{ManPum2,LSV}.
\vskip .1cm
- There exist various sets of conditions under which the existence of an invariant measure $\mu$ for the skew product $F$ that satisfies (r6) is guaranteed. See e.g.~\cite{pelikan,morita1,GB03,BG06,Ino} for results in this direction where it is assumed that the random system under consideration is expanding on average. For the random map composed of LSV maps, i.e.~the Manneville-Pomeau type maps introduced in \cite{LSV}, results about the existence of an invariant measure $\mu$ that satisfies (r6) can be found in e.g.~\cite{bahsoun3,dijk,zeegers}.
\vskip .1cm
- The results from \cite{KM21} give an algorithm for determining explicit formulae for invariant probability measures of the form $\mathbf P \times \rho$ with $\rho \ll \lambda$ in case all maps $T_i$ are piecewise linear Lasota-Yorke type maps satisfying some further conditions. Having an explicit formula facilitates the computation of the entropy of $\Delta$ and the verification of (r7).

\begin{remark}{\rm
If $I$ consists of only one element, then the random number system reduces to an interval map. In this case, conditions (r2), (r5) and (r7) are equivalent with assuming that (n1), (n3) and (n4) hold for this interval map, respectively. Moreover, it follows from (r6) that the interval map is onto $[0,1)$ up to some $\lambda$-measure zero set and it follows from (r6) that this interval map satisfies (n2) except that the density does not necessarily have to be bounded and bounded away from zero. Thus in particular, each NTFM is an example of a random number system where the index set consists of only one element. Hence, Theorem \ref{t:randomlochs} is an extension of the result in \eqref{q:df} from \cite{dajani01} and shows that \eqref{q:df} remains true for two NTFM's for which the condition in (n2) on the bounds on the density are dropped.
}\end{remark}

\subsection{Invertible base maps and invariant measures}

The dynamics on the base space $I$ of a random number system is given by the left shift $\tau$ on the set $I^\mathbb N$, which is not invertible. This setup corresponds to the setup for random systems associated to number expansions that is adopted in most of the references mentioned in the introduction. To prove Theorem~\ref{t:randomlochs}, however, we employ known theory on random systems and fiber entropy that is available for skew products with invertible dynamics on the base space. One can easily extend the one-sided shift in the first coordinate of $F$ to a two-sided (thus invertible) shift and as we shall see next, this has no profound effect on the invariant measures.
\vskip .2cm

Let $\hat \tau$ denote the left shift on $I^\mathbb Z$ and extend the skew product $F$ to a map $\hat F$ that is invertible in the first coordinate by setting
\[ \hat F: I^{\mathbb{Z}} \times [0,1) \rightarrow I^{\mathbb{Z}} \times [0,1),  \,(\hat \omega,x) \mapsto (\hat \tau (\hat \omega), T_{\hat \omega_1} (x)).\]
Let $\mathcal B_I^{\mathbb Z}$ denote the Borel $\sigma$-algebra on $I^\mathbb Z$. Use $\pi: I^\mathbb Z \to I^\mathbb N$ to denote the canonical projection. To keep notation simple for two-sided sequences $\hat \omega \in I^\mathbb Z$ and $n \ge 0$ we use the same notation for $T_{\hat \omega}^n$, $\alpha_{\hat \omega,n}$ and $\mu_{\hat \omega}$ as for one-sided sequences, i.e., 
\[ T_{\hat \omega}^n = T_{\pi(\hat \omega)}^n, \quad \alpha_{\hat \omega,n} = \alpha_{\pi(\hat \omega),n}, \quad \mu_{\hat \omega} = \mu_{\pi(\hat \omega)}.\]
The skew product $\hat F$ is measurable due to (r1). The next proposition gives a relation between the (ergodic) invariant measures of $F$ and those of $\hat F$. It can be found in a slightly more restrictive setting in Appendix A of \cite{ale-jan} (see also the references therein), but the proof carries over unchanged to our setting. We reproduce the statement here for our setting for convenience. Use $\hat\pi_I: I^\mathbb Z \times [0,1) \to I^\mathbb Z$ and $\Pi: I^\mathbb Z \times [0,1) \to I^\mathbb N \times [0,1)$ to denote the respective canonical projections.

\begin{prop}{\cite[Proposition A.1 and Remark A.4]{ale-jan}}\label{p:alejan}
Let $\mu$ be an $F$-invariant probability measure with marginal $\nu = \mu \circ \pi_I^{-1}$ and system of conditional measures $\{\mu_{\omega}\}_{\omega \in I^\mathbb N}$. Then the following statements hold.
\begin{itemize}
\item[(i)] There exists an $\hat F$-invariant probability measure $\hat \mu$ with marginal $\hat{\nu} = \hat \mu \circ \hat\pi_I^{-1}$ and a system of conditional measures $\{\hat \mu_{\hat \omega}\}_{\hat \omega \in I^\mathbb Z}$ such that, for $\hat{\nu}$-a.a.~$\hat \omega \in I^{\mathbb{Z}}$,
\[ \hat \mu_{\hat \omega}(B) = \lim_{n \rightarrow \infty} \mu_{\hat \tau^{-n} \hat \omega}\Big((T^n_{\hat \tau^{-n} \hat \omega})^{-1}(B)\Big), \qquad B \in \mathcal{B}.\]
\item[(ii)] Conversely, let $\hat \mu$ be an $\hat F$-invariant probability measure with marginal $\hat \nu = \hat \mu \circ \hat\pi_I^{-1}$. Then the probability measure
\[ \tilde \mu = \hat \mu \circ \Pi^{-1}\]
is $F$-invariant and has marginal $\tilde \nu = \hat \nu \circ \pi^{-1}$.
\item[(iii)] The correspondence $\mu \leftrightarrow \hat\mu$ given by (i) and (ii) is one-to-one and has the property that $\mu$ is ergodic for $F$ if and only if $\hat{\mu}$ is ergodic for $\hat F$.
\end{itemize}
\end{prop}

From Proposition \ref{p:alejan} and (r6) we obtain an $\hat F$-invariant and ergodic probability measure $\hat{\mu}$ with a system of conditional measures $\{\hat{\mu}_{\hat \omega}\}_{\hat \omega \in I^\mathbb Z}$ for the marginal $\hat{\nu} = \hat \mu \circ \hat\pi_I^{-1}$. The following lemma will be used later.

\begin{lemma}\label{l:absct}
For $\hat \nu$-a.e.~$\hat \omega \in I^\mathbb Z$ it holds that $\hat \mu_{\hat \omega} \ll \lambda$. Moreover, $\hat \mu \ll \hat \nu \times \lambda$ and for $\hat \nu \times \lambda$-a.e.~$(\hat \omega,x) \in I^\mathbb Z \times [0,1)$ we have
\[ \frac{d\hat \mu}{d\hat \nu \times \lambda} (\hat \omega,x) = \frac{d\hat \mu_{\hat \omega}}{d\lambda} (x).\]
\end{lemma}

\begin{proof}
Combining that $\tau$ is non-singular with respect to $\mathbf{P}$ with (r4) gives that for each $n \ge 1$,
\[ \mathbf{P} ( \{ \omega \in I^\mathbb N \, : \,  \exists B \in \mathcal B \quad  \lambda (B) =0 \, \text{ and } \,\lambda ((T_\omega^n)^{-1}B)>0 \} ) =0.\]
Recall that $\nu$ and $\mathbf P$ are equivalent. Furthermore, the measure $\hat \nu$ is invariant with respect to $\hat \tau$ and $\hat \nu \circ \pi^{-1} = \nu$. This gives for each $m \in \mathbb Z$ that
\[ \hat \nu (\hat \tau^{-m} \pi^{-1} \{ \omega \in I^\mathbb N \, : \,  \exists B \in \mathcal B \quad  \lambda (B) =0 \, \text{ and } \,\lambda ((T_\omega^n)^{-1}B)>0 \} ) =0.\]
Taking $m= -n$ it follows for $\hat \nu$-a.e.~$\hat \omega \in I^\mathbb Z$ and $n \geq 1$ that
\begin{equation}\label{q:lambdaTshift}
\lambda(B)=0 \quad \Rightarrow \quad \lambda((T_{\hat \tau^{-n}(\hat \omega)}^n)^{-1}B)=0, \qquad \forall B \in \mathcal{B}.
\end{equation}
The measures $\mu$ and $\nu \times \lambda$ are equivalent by condition (r6). Let $A \in \mathcal B_I^\mathbb N$ and $B \in \mathcal B$. Then by \eqref{q:system}
\[ \int_A \mu_\omega (B) \, d \nu(\omega)  =  \mu(A \times B) = \int_A \int_B \frac{d\mu}{d\nu \times \lambda} (\omega,x) \, d\lambda(x) \, d\nu(\omega),\]
which means that for any $B \in \mathcal B$ we can find a $\nu$-null set $N_B$, such that for all $\omega \in I^\mathbb N \setminus N_B$ we have
\[ \mu_\omega (B) = \int_B \frac{d\mu}{d\nu \times \lambda} (\omega,x) \, d\lambda(x).\]
Since $\mathcal B$ is countably generated, we can find a $\nu$-null set $N$, such that for each $\omega \in I^\mathbb N \setminus N$ and all $B \in \mathcal B$,
\[ \mu_\omega (B) = \int_B \frac{d\mu}{d\nu \times \lambda} (\omega,x) \, d\lambda(x).\]
Hence, $\mu_\omega \ll \lambda$ for $\nu$-a.e.~$\omega \in I^\mathbb N$ and for those $\omega$,
\[ \frac{d\mu_\omega}{d \lambda} (x) = \frac{d\mu}{d\nu \times \lambda} (\omega,x) \qquad \lambda\text{-a.e.}\]
It immediately follows that $\hat \nu$-a.e.~$\hat \omega \in I^\mathbb Z$ satisfies $\mu_{\hat \omega}\ll \lambda$ and since $\hat \nu$ is $\hat \tau$-invariant, we get that for $\hat \nu$-a.e.~$\hat \omega \in I^\mathbb Z$, $\mu_{\hat \tau^{-n} \hat \omega} \ll \lambda$ for each $n$. Combining this with \eqref{q:lambdaTshift} and Proposition~\ref{p:alejan}(i) gives that $\hat \mu_{\hat \omega} \ll \lambda$ for $\hat \nu$-a.e.~$\hat \omega \in I^\mathbb Z$. Since $\{ \hat \mu_{\hat \omega} \}$ is a system of conditional invariant measures, we get that for each $A \in \mathcal B_I^\mathbb Z \times \mathcal B$ (cf.~\eqref{q:system}),
 \[  \hat \mu(A) = \int_{I^\mathbb Z } \hat \mu_{\hat \omega} (A_{\hat \omega}) \, d \hat \nu (\hat \omega)
= \int_A \frac{d\hat \mu_{\hat \omega}}{d\lambda} (x) \, d \hat \nu \times \lambda(\hat \omega,x),\]
where $A_{\hat \omega} = \{ x \in [0,1) \, : \, (\hat \omega,x) \in A \}$. This means that $\hat \mu \ll \hat \nu \times \lambda$ and that for $\hat \nu \times \lambda$-a.e.~$(\hat \omega,x)$ it holds that $\frac{d\hat \mu}{d \hat \nu \times \lambda} (\hat \omega,x) =  \frac{d\hat \mu_{\hat \omega}}{d\lambda} (x)$.
\end{proof}

\subsection{Fiber entropy}
The concept of fiber entropy was introduced in \cite{Abramov2}. Here, as well as in the later works \cite{Branton82} and \cite{morita3}, the entropy of a skew product is studied for the case that the associated transformations are all measure-preserving with respect to the same measure. In \cite{kifer} and \cite{ledrappier88}, the notion fiber entropy is considered for skew products of transformations with a Bernoulli measure on the base space. These two settings are extended in the works \cite{bogenschutz92a} and \cite{bogenschutz92}, where the invariant measure of the skew product admits a system of conditional measures. 

\vskip .2cm
Here we introduce fiber entropy for random number systems following the approach of Bogensch\"utz \cite{Bogenschutz93}. Two standing assumptions in \cite{Bogenschutz93} (and one of the reasons why we extended $F$ to $\hat F$) are that the dynamics on the base space are invertible and that the $\sigma$-algebra considered on the first coordinate is countably generated. By our definition of $\hat F$ and the assumption that $I$ is a Polish space, we satisfy both these assumptions.
 
\vskip .2cm
Consider the sub-$\sigma$-algebra $\mathcal A := \mathcal B_I^\mathbb Z \times [0,1)$ of $\mathcal B_I^\mathbb Z \times \mathcal B$. From $\hat \tau \circ \hat \pi_I = \hat \pi_I \circ \hat F$ we see that $\hat F^{-1} \mathcal A \subseteq \mathcal A$ and in this situation we can define the {\em conditional entropy} of a partition $\mathcal P$ of $I^\mathbb Z \times [0,1)$ given $\mathcal A$ as in \cite{Bogenschutz93} using the conditional expectation by
\[ H_{\hat \mu} (\mathcal P | \mathcal A) = - \int_{I^\mathbb Z \times [0,1)} \sum_{P \in \mathcal P} \mathbb E_\mu (1_P | \mathcal A) \log \mathbb E_\mu (1_P | \mathcal A) \, d\mu.\]
Then $H_{\hat \mu}(\mathcal P | \mathcal A) \le H_{\hat \mu} (\mathcal P)$ (see e.g.~\cite[Lemma 1.2(vi)]{kifer}). The {\em fiber entropy} of the random number system $\mathcal T$ is defined in \cite[Definition 2.2.1]{Bogenschutz93} as
\[ h^{\text{fib}}(\mathcal T) := \sup_{\mathcal P} \lim_{n \to \infty} \frac1n H_{\hat \mu} \Big( \bigvee_{k=0}^{n-1} \hat F^{-k} \mathcal P | \mathcal A \Big),\]
where the supremum is taken over all partitions $\mathcal P$ satisfying $H_{\hat \mu} (\mathcal P | \mathcal A) < \infty$.

\vskip .2cm
As usual for entropy it is often not very practical to compute the fiber entropy of a system directly from the definition. One way to determine $h^{\text{fib}}(\mathcal T)$ follows from the main theorem of \cite{bogenschutz92}, which gives the {\em Abramov-Rokhlin formula}
\begin{equation}\label{q:AR}
h_{\hat \mu}(\hat F) = h_{\hat \nu}(\hat \tau) + h^{\text{fib}}(\mathcal T).
\end{equation}
This leads to an expression for $h^{\text{fib}}(\mathcal T)$ in case $h_{\hat \nu}(\hat \tau) < \infty$. Furthermore, the literature provides two versions of the Kolmorogov-Sinai Theorem for random systems, namely \cite[Lemma II.1.5]{kifer} and \cite[Theorem 2.3.3]{Bogenschutz93}. The first of these results requires a generating partition $\mathcal P$ of the product space $I^\mathbb Z \times [0,1)$ in the sense that
\begin{equation}\label{q:kifergen}
\sigma \Big( \bigvee_{k\ge 0} F^{-k} \mathcal P \Big) \vee \mathcal A = \mathcal B_I^\mathbb Z \times \mathcal B.
\end{equation}
The latter requires a partition $\gamma$ of $[0,1)$ that satisfies
\[ \sigma \Big(  \bigvee_{k\ge 0} (T_{\hat \omega}^k)^{-1} \gamma \Big)= \mathcal B  \qquad \hat \nu\text{-a.e.}\]
For random number systems a natural candidate for a generating partition is provided by the family of partitions $\{ \alpha_i\}_{i \in I}$ and the corresponding partition $\Delta$ on $I^\mathbb N \times [0,1)$ from \eqref{q:delta}. $\Delta$ is easily extended to a partition $\hat \Delta$ of $I^{\mathbb{Z}} \times [0,1)$ by setting
\[ \hat{\Delta} = \{\Pi^{-1} \Delta(j): j \ge 0\}.\]
The property (r5) is not enough to guarantee that the conditions of \cite[Lemma II.1.5]{kifer} are satisfied by $\hat \Delta$ and the conditions of \cite[Theorem 2.3.3]{Bogenschutz93} are not satisfied either, since we consider a family of partitions $\{ \alpha_i\}_{i \in I}$ on $[0,1)$ rather than a single partition $\gamma$. Nevertheless, since from (r7) we have
\begin{equation}\label{q:entropygivenA}
H_{\hat \mu} (\hat \Delta | \mathcal A) \le H_{\hat \mu} (\hat \Delta) = H_\mu (\Delta) < \infty,
\end{equation}
and from Lemma~\ref{l:absct} we know that $\hat \mu_{\hat \omega} \ll \lambda$ holds for $\hat \nu$-a.e.~$\hat \omega \in I^\mathbb Z$, by a reasoning completely analogous to the proof of \cite[Theorem 2.3.3]{Bogenschutz93} we still obtain that
\begin{equation}\label{limit6}
h^{\text{fib}}(\mathcal T) = \lim_{n \to \infty} \frac1n H_{\hat \mu} \Big( \bigvee_{k=0}^{n-1} \hat F^{-k} \hat \Delta | \mathcal A \Big).
\end{equation}
Hence, $\hat \Delta$ serves enough as a generating partition to have a Kolmogorov-Sinai type result and therefore an expression of $h^{\text{fib}}(\mathcal T)$ in terms of $\hat \Delta$. The reason we have put (r5) as a property of random number systems and not a condition like \eqref{q:kifergen} is that compared to condition \eqref{q:kifergen} from \cite[Lemma II.1.5]{kifer}, condition (r5) is easier to verify.

\vskip .2cm
The sequence $(a_n)_{n \in \mathbb{N}}$ given by $a_n = H_{\hat \mu} \Big( \bigvee_{k=0}^{n-1} \hat F^{-k} \hat \Delta | \mathcal A \Big)$ is subadditive (see e.g.~Theorem II.1.1 in \cite{kifer}), thus it follows from Fekete's Subadditive Lemma together with \eqref{limit6} that
\[ h^{\text{fib}}(\mathcal T) = \inf_{n \in \mathbb{N}} \frac1n H_{\hat \mu} \Big( \bigvee_{k=0}^{n-1} \hat F^{-k} \hat \Delta | \mathcal A \Big).\]
In particular, (r7) implies via \eqref{q:entropygivenA} that any random number system $\mathcal T$ has $h^{\text{fib}}(\mathcal T) < \infty$.

\vskip .2cm
Using standard arguments (see e.g.~\cite[Lemma 9.1.12]{foundations}) one can deduce that
\[ \lim_{n \to \infty} \frac1n H_{\hat \mu} \Big( \bigvee_{k=0}^{n-1} \hat F^{-k} \hat \Delta | \mathcal A \Big) = \lim_{n \to \infty} H_{\hat \mu} \Big( \hat \Delta | \sigma \Big( \bigvee_{k=1}^{n-1} \hat F^{-k} \hat \Delta \Big) \vee \mathcal A \Big).\]
This leads to the following expression for the fiber entropy of a random number system $\mathcal T$:
\[ h^{\text{fib}}(\mathcal{T}) = \lim_{n \rightarrow \infty} H_{\hat \mu }\Big(\hat \Delta |\sigma\Big(\bigvee_{k=1}^{n-1} \hat F^{-k} \hat{\Delta}\Big) \vee \mathcal{A}\Big).\]

\vskip .2cm
For a partition $\mathcal P$ of $I^{\mathbb{Z}} \times [0,1)$ we can for each $\hat{\omega} \in I^{\mathbb{Z}}$ obtain a partition $\mathcal P_{\hat{\omega}}$ of $[0,1)$ by intersecting it with the fiber $\{\hat{\omega}\} \times [0,1)$, i.e.,~$\mathcal P_{\hat{\omega}} = \{Z_{\hat{\omega}}: Z \in \mathcal P \}$ where $Z_{\hat{\omega}} = \{x \in [0,1): (\hat{\omega},x) \in Z\}$. With this notation, note that
\[ \Big( \bigvee_{k=0}^{n-1}\hat F^{-k} \hat \Delta \Big)_{\hat \omega} = \alpha_{\hat \omega,n} \quad \text{ and } \quad  \Big( \bigvee_{k=1}^{n-1}\hat F^{-k} \hat \Delta \Big)_{\hat \omega} = T_{\hat{\omega}_1}^{-1} \alpha_{\hat{\tau} \hat{\omega},n}.\]
It now follows from \cite[Lemma 2.2.3]{Bogenschutz93} that
\begin{align}\label{eq20a}
 H_{\hat \mu} \Big( \bigvee_{k=0}^{n-1} \hat F^{-k} \hat \Delta | \mathcal A \Big) = \int_{I^{\mathbb{Z}}} H_{\hat{\mu}_{\hat \omega}}(\alpha_{\hat \omega,n}) \, d\hat{\nu}(\hat \omega)
\end{align}
and
\[ H_{\hat \mu }\Big(\hat \Delta |\sigma\Big(\bigvee_{k=1}^{n-1} \hat F^{-k} \hat{\Delta}\Big) \vee \mathcal{A}\Big) = \int_{I^\mathbb Z} H_{\hat \mu_{\hat \omega}}\big(\alpha_{\hat \omega_1} \Big| \sigma(T_{\hat \omega_1}^{-1} \alpha_{\hat\tau \hat \omega,n})\big) d\hat \nu(\hat \omega).\]
Hence, $h^{\mathrm{fib}}(\mathcal{T})$ can be rewritten as
\begin{align}
h^{\mathrm{fib}}(\mathcal{T}) =\ &  \lim_{n \rightarrow \infty} \frac{1}{n} \int_{I^{\mathbb{Z}}} H_{\hat{\mu}_{\hat \omega}}(\alpha_{\hat \omega,n}) \, d\hat{\nu}(\hat \omega) \label{eq22a}\\
=\ & \lim_{n \rightarrow \infty} \int_{I^\mathbb Z} H_{\hat \mu_{\hat \omega}}\big(\alpha_{\hat \omega_1} \Big| \sigma(T_{\hat \omega_1}^{-1} \alpha_{\hat\tau \hat \omega,n})\big) d\hat \nu(\hat \omega) \label{eq22b}.
\end{align}

\begin{remark}\label{r7}{\rm
Condition (r7) ensures that $H_{\hat \mu} (\hat \Delta | \mathcal A)  < \infty$ as follows from \eqref{q:entropygivenA}, which enables us to rewrite $h^{\mathrm{fib}}(\mathcal{T})$ as in \eqref{limit6}. It follows from \eqref{eq20a} that $H_{\hat \mu} (\hat \Delta | \mathcal A)  < \infty$ is also ensured by requiring
\begin{align}\label{eq24a}
\int_{I^{\mathbb{Z}}} H_{\hat{\mu}_{\hat \omega}}(\alpha_{\hat \omega_1}) \, d\hat{\nu}(\hat \omega) < \infty,
\end{align}
which is in general weaker than condition (r7). In view of this, we will call a collection $\mathcal{T}=(I, \mathbf P, \{T_i\}_{i \in I},\mu, \{\alpha_i\}_{i \in I})$ on $[0,1)$ satisfying (r1)-(r6) together with \eqref{eq24a} a random number system as well. All our results also hold in this case.
}\end{remark}

\begin{remark}{\rm
Note that in case $I = \{1\}$ contains only one element, then $h_{\nu}(\tau)=0$ and Theorem~\ref{t:main}(ii) gives that $h^{\text{fib}}(\mathcal T) = h_\rho(T_1)$, where $\rho$ is the ergodic invariant probability measure for $T_1$ equivalent to $\lambda$.
}\end{remark}

\section{Proofs of the main results}\label{s:main}

\subsection{Asymptotic size of cylinder sets and the random Rokhlin formula}
In this section we prove the three statements of Theorem~\ref{t:main} and we obtain Theorem~\ref{t:randomlochs} as a corollary. We start with Theorem~\ref{t:main}(i) on the asymptotic size of cylinder sets. Let $\mathcal T = (I, \mathbf P, \{T_i\}_{i \in I},\mu, \{ \alpha_i\}_{i \in I})$ be a random number system.
\begin{theorem}\label{prop2.6a}
For $\mathbf{P}$-a.a.~$\omega \in I^{\mathbb{N}}$ we have
\[ \lim_{n \rightarrow \infty} -\frac{\log \lambda\big(\alpha_{\omega,n}(x)\big)}{n} = h^{\mathrm{fib}}(\mathcal{T}) \qquad \lambda\text{-a.e.}\]
\end{theorem}

\begin{proof}
Since $I^{\mathbb{Z}}$ is a Polish space on which $\hat \tau$ is invertible and the measure $\hat \mu$ is $\hat{F}$-invariant and ergodic, we are in the position to apply \cite[Proposition 2.2 (3)]{zhu08}, which gives the Shannon-McMillan-Breiman Theorem for random dynamical systems. We apply it to the partition $\hat \Delta$ of $I^\mathbb Z \times [0,1)$, which is an at most countable partition satisfying $H_{\hat \mu} (\hat \Delta | \mathcal A) < \infty$. Note that \cite[Proposition 2.2]{zhu08} considers random compositions of continuous transformations and requires the partition under consideration to be finite. However, the continuity of the transformations is not used in the proof, and the necessary condition on the partition is that it is at most countable and satisfies
\begin{equation}\label{q:zhu}
\int_{I^\mathbb Z} H_{\hat \mu_{\hat \omega}} (\hat \Delta_{\hat \omega}) \, d \hat \nu(\hat \omega) < \infty .
\end{equation}
From \eqref{eq20a} it follows that $\int_{I^\mathbb Z} H_{\hat \mu_{\hat \omega}} (\hat \Delta_{\hat \omega}) \, d \hat \nu(\hat \omega) = H_{\hat \mu} (\hat \Delta | \mathcal A)$, so the condition from \eqref{q:zhu} is satisfied. Thus, using the formula for $h^{\text{fib}}(\mathcal{T})$ from \eqref{eq22a}, \cite[Proposition 2.2 (3)]{zhu08} gives that
\begin{equation}\label{eq:2.42a}
\lim_{n \rightarrow \infty} - \frac{\log \hat \mu_{\hat{\omega}}\big( \alpha_{\hat \omega,n}(x)\big)}{n} = h^{\text{fib}}(\mathcal{T}), \qquad \hat \mu\text{-a.e.~$(\hat{\omega},x) \in I^\mathbb Z \times [0,1)$}.
\end{equation}

Let $C \in  \mathcal B_I^\mathbb Z \times \mathcal B$ be the set of points $(\hat \omega,x)$ with the following three properties:
\begin{itemize}
\item[(i)] the limit statement from \eqref{eq:2.42a} holds;
\item[(ii)] $\frac{d\hat \mu}{d\hat \nu \times \lambda} (\hat \omega,x) = \frac{d\hat \mu_{\hat \omega}}{d\lambda} (x)$;
\item[(iii)] $x$ is a Lebesgue density point for $\frac{d\hat \mu_{\hat \omega}}{d\lambda}$.
\end{itemize}
Since $\frac{d\hat \mu_{\hat \omega}}{d\lambda} \in L^1(\lambda)$ holds for $\hat \nu$-a.e.~$\hat \omega$, it follows from the Lebesgue Differentiation Theorem that for such $\hat \omega$ we have that $\lambda$-a.e.~$x \in [0,1)$ is a Lebesgue density point of $\frac{d\hat \mu_{\hat \omega}}{d\lambda}$. Together with \eqref{eq:2.42a} and Lemma~\ref{l:absct} we deduce that $\hat \mu (C)=1$. Define
\[ A = \Big\{(\omega,x) \in I^{\mathbb{N}} \times [0,1) : \frac{d\hat{\mu}}{d\hat{\nu} \times \lambda}\Big|_{\Pi^{-1}\{ (\omega,x)\} \cap C} = 0\Big\}.\]
Then by Proposition~\ref{p:alejan},
\[ \mu(A) =  \hat \mu \big(\Pi^{-1}A \cap C \big)   = \int_{\Pi^{-1}A \cap C} \frac{d\hat \mu}{d\hat \nu \times \lambda} \, d\hat \nu \times \lambda = 0,\]
thus $\mathbf P \times \lambda(A) = 0$ by (r6). Hence, for $\mathbf P \times \lambda$-a.e.~$(\omega,x)$ there exists an $\hat{\omega} \in I^{\mathbb{Z}}$ such that $(\hat{\omega},x) \in C$ and $\pi(\hat{\omega}) = \omega$ and $\frac{d\hat{\mu}_{\hat{\omega}}}{d\lambda}(x) > 0$, where this last part follows from property (ii) of $C$. For each such $(\omega,x)$ and $\hat{\omega}$ it follows from property (iii) of elements in $C$ that for every $\varepsilon > 0$ there exists an $N \in \mathbb{N}$ such that for each integer $n \geq N$,
\[ \Big(\frac{d\hat{\mu}_{\hat{\omega}}}{d\lambda}(x) - \varepsilon\Big) \lambda\big(\alpha_{\hat \omega,n}(x)\big) \leq \hat{\mu}_{\hat{\omega}}\big(\alpha_{\hat \omega,n}(x)\big) \leq \Big(\frac{d\hat{\mu}_{\hat{\omega}}}{d\lambda}(x) + \varepsilon\Big) \lambda\big(\alpha_{\hat \omega,n}(x)\big).
\]
Combining this with \eqref{eq:2.42a} yields the result.
\end{proof}

From this result the Random Lochs' Theorem from Theorem~\ref{t:randomlochs} immediately follows.
\begin{proof}[Proof of Theorem~\ref{t:randomlochs}]
We have $h^{\mathrm{fib}}(\mathcal{T}), h^{\mathrm{fib}}(\mathcal{S}) \in (0,\infty)$, so combining Theorem~\ref{t:partitions} with Theorem~\ref{t:main}(i) yields the desired result.
\end{proof}

The last item of Theorem~\ref{t:main} is the random Rokhlin formula relating fiber entropy to the Jacobian of the transformations. We first prove an auxiliary lemma, for which we introduce some notation.

\vskip .2cm
Assume as in Theorem~\ref{t:main}(iii) that for each $i \in I$ and $A \in \alpha_i$ the restriction $T_i|_A$ is differentiable. Then for each $i \in I$ the Jacobian $JT_i$ of $T_i$ with respect to $\lambda$ exists and is equal to $J T_i = |DT_i(x)|$ for $\lambda$-a.e.~$x \in [0,1)$. From Lemma~\ref{l:absct} together with the $\hat \tau$-invariance of $\hat{\nu}$ it follows that for $\hat \nu$-a.e.~$\hat \omega \in I^\mathbb Z$, $\hat{\mu}_{\hat \tau \hat \omega} \ll \lambda$. By (r5) we obtain that for $\hat \nu$-a.e.~$\hat \omega \in I^\mathbb Z$,
\begin{equation}\label{q:hatmuhatsigmazero}
 \sigma\big(\lim_{n \rightarrow \infty} \alpha_{\hat \tau \hat \omega,n}) = \mathcal{B} \, \text{ up to sets of $\hat{\mu}_{\hat \tau \hat \omega}$-measure zero}.
\end{equation}
Since the standing assumptions of \cite{Bogenschutz93} are satisfied, \cite[Lemma 1.1.2]{Bogenschutz93} gives that
\begin{equation}\label{q:abscthatsigma}
\hat{\mu}_{\hat \omega} \circ T_{\hat \omega_1}^{-1} = \hat{\mu}_{\hat \tau \hat \omega}, \qquad \hat \nu\text{-a.a.~$\hat \omega \in I^{\mathbb{Z}}$}.
\end{equation}
Let
\[ E = \{ \hat \omega \in I^\mathbb Z \, : \, \hat{\mu}_{\hat \tau \hat \omega} \ll \lambda \, \text{ and } \text{ \eqref{q:hatmuhatsigmazero} and \eqref{q:abscthatsigma} hold} \}.\]
Then $\hat \nu (E)=1$. For $\hat \omega \in E$ let $C_{\hat \omega} = \{ x \in [0,1): \frac{d\hat{\mu}_{\hat \tau \hat \omega}}{d\lambda}(x) > 0\}$, define $h_{\hat \omega}: [0,1) \rightarrow [0, \infty)$ by
\[ h_{\hat \omega}(y) = \begin{cases}
\big(\frac{d\hat \mu_{\hat \tau \hat \omega}}{d\lambda}(y) \big)^{-1} , & \text{if } y \in C_{\hat \omega},\\
0, & \text{if } y \in [0,1) \backslash C_{\hat \omega},
\end{cases}\]
and define
\[ \mathcal L_{\hat \omega}\psi (y) = h_{\hat \omega}(y) \sum_{z \in T_{\hat \omega_1}^{-1}\{y\}} \frac{\psi(z)}{J T_{\hat \omega_1}(z)} \cdot \frac{d\hat \mu_{\hat \omega}}{d\lambda}(z),\]
which, as we will see in the proof of Lemma \ref{l:operator} below, is well-defined as an operator from $L^1(\hat \mu_{\hat \omega})$ to $L^1 (\hat \mu_{\hat \tau \hat \omega})$.

\begin{lemma}\label{l:operator}
For $\hat \omega \in E$, $\psi \in L^1(\hat \mu_{\hat \omega})$ and $n \ge 1$ the following hold.
\begin{itemize}
\item[(i)] $\mathbb E_{\hat \mu_{\hat \omega}}(\psi| \sigma(T_{\hat \omega_1}^{-1} \alpha_{\hat \tau \hat \omega,n}))(x) = \mathbb{E}_{\hat{\mu}_{\hat \tau \hat \omega}}(\mathcal L_{\hat \omega}\psi |\sigma(\alpha_{\hat \tau \hat \omega,n}))(T_{\hat \omega_1} x)$ {for $\hat{\mu}_{\hat  \omega}$-a.e.~$x$.}
\vskip .1cm
\item[(ii)] $\lim_{n \rightarrow \infty} \mathbb E_{\hat \mu_{\hat \tau \hat \omega}}(\mathcal L_{\hat \omega}\psi |\sigma(\alpha_{\hat\tau \hat \omega,n})) = \mathbb E_{\hat \mu_{\hat \tau \hat \omega}}(\mathcal L_{\hat \omega}\psi |\mathcal{B}) = \mathcal L_{\hat \omega}\psi$  {for $\hat{\mu}_{\hat \tau \hat  \omega}$-a.e.~$x$.}
\end{itemize}
\end{lemma}

\begin{proof}
By \eqref{q:abscthatsigma} and the definition of $C_{\hat \omega}$, $\hat \mu_{\hat \omega}(T_{\hat \omega_1}^{-1}C_{\hat \omega}) = \hat \mu_{\hat \tau \hat \omega}(C_{\hat \omega}) = 1$. Using a change of variables formula we obtain for all $B \in \alpha_{\hat \tau \hat \omega,n}$ that
\[ \begin{split}
\int_B \mathcal L_{\hat \omega}\psi  \, d\hat \mu_{\hat\tau \hat \omega} = \int_{B\cap C_{\hat \omega}} \mathcal L_{\hat \omega}\psi  \frac{d\hat{\mu}_{\hat \tau \hat \omega}}{d\lambda} \, d\lambda =\ & \int_{B \cap C_{\hat \omega}} \sum_{z \in T_{\hat \omega_1}^{-1}\{y\}} \frac{\psi(z)}{J T_{\hat \omega_1}(z)} \cdot \frac{d\hat{\mu}_{\hat \omega}}{d\lambda}(z) \, d\lambda(y) \\
=\ & \sum_{A \in \alpha_{\hat \omega_1}} \int_{T_{\hat \omega_1}(A) \cap B \cap C_{\hat \omega}} \Big(\frac{\psi \cdot \frac{d\hat \mu_{\hat \omega}}{d\lambda}}{J T_{\hat \omega_1}}\Big) \circ (T_{\hat \omega_1}|_A)^{-1}(y) \, d\lambda(y)\\
=\ & \sum_{A \in \alpha_{\omega_1}} \int_{A \cap T_{\hat \omega_1}^{-1} B \cap T_{\hat \omega_1}^{-1} C_{\hat \omega}} \psi \cdot \frac{d\hat \mu_{\hat \omega}}{d\lambda}  \, d\lambda\\
=\ & \int_{T_{\hat \omega_1}^{-1} B} \psi \, d\hat \mu_{\hat \omega}.
\end{split}\]
In particular,  $\mathcal L_{\hat \omega}: L^1(\hat \mu_{\hat \omega}) \rightarrow L^1 (\hat \mu_{\hat \tau \hat \omega})$ is well-defined. We obtain that
\[ \begin{split}
\mathbb{E}_{\hat{\mu}_{\hat \tau \hat \omega}}(\mathcal L_{\hat \omega}\psi |\sigma(\alpha_{\hat \tau \hat \omega,n}))(T_{\hat \omega_1} x) =\ & \sum_{B \in \alpha_{\hat \tau \hat \omega,n}} 1_B(T_{\omega_1}x) \frac{\int_B \mathcal L_{\hat \omega}\psi \, d\hat{\mu}_{\hat \tau \hat \omega}}{\hat{\mu}_{\hat \tau \hat \omega}(B)}\\
=\ & \sum_{B \in \alpha_{\hat \tau \hat \omega,n}} 1_{T_{\hat \omega_1}^{-1}B}(x) \frac{\int_{T_{\hat \omega_1}^{-1}B} \psi \, d\hat \mu_{\hat \omega}}{\hat \mu_{\hat \omega}(T_{\hat \omega_1}^{-1}B)}\\
=\ & \sum_{A \in T_{\hat \omega_1}^{-1} \alpha_{\hat \tau \hat \omega,n}} 1_A (x) \frac{\int_A \psi \, d\hat \mu_{\hat \omega}}{\hat \mu_{\hat \omega}(A)}\\
= \ & \mathbb E_{\hat \mu_{\hat \omega}}(\psi| \sigma(T_{\hat \omega_1}^{-1} \alpha_{\hat \tau \hat \omega,n}))(x)
\end{split}\]
for $\hat{\mu}_{\hat \tau \hat  \omega}$-a.e.~$x$. This gives (i). Part (ii) follows from combining \eqref{q:hatmuhatsigmazero} and L\'evy's Upward Theorem.
\end{proof}

This lemma is enough to prove the following Random Rokhlin's Formula.

\begin{theorem}\label{prop2.7a}
If for each $i \in I$ and $A \in \alpha_i$ the restriction $T_i|_A$ is differentiable, then
\[ h^{\mathrm{fib}}(\mathcal{T}) = \int_{I^\mathbb N \times [0,1)} \log |DT_{\omega_1}(x)| \, d\mu(\omega,x). \]
\end{theorem}

\begin{proof}
We follow the reasoning of the proof of \cite[Theorem 9.7.3]{foundations}. Write $\phi(x) = -x \log x$. As before, set $\mathcal{A} = \mathcal{B}_I^{\mathbb{Z}} \times [0,1)$. Then using \eqref{q:zhu} and the Dominated Convergence Theorem, we see from \eqref{eq22b} that
\begin{equation}\label{q:hfibfirst}
\begin{split}
h^{\text{fib}}(\mathcal{T}) =\ & \int_{I^\mathbb Z} \lim_{n \rightarrow \infty} H_{\hat \mu_{\hat \omega}}\big(\alpha_{\hat \omega_1} \Big| \sigma(T_{\hat \omega_1}^{-1} \alpha_{\hat\tau \hat \omega,n})\big) d\hat \nu(\hat \omega)\\
=\ & \int_{I^\mathbb Z} \lim_{n \rightarrow \infty} \int_{[0,1)} \sum_{A \in \alpha_{\omega_1}} \phi\big(\mathbb{E}_{\hat{\mu}_{\hat \omega}}(1_A|\sigma(T_{\hat \omega_1}^{-1} \alpha_{\hat \tau \hat \omega,n}))\big) (x)\, d\hat{\mu}_{\hat \omega} (x) d\hat{\nu}(\hat \omega)\\
=\ & \int_{I^\mathbb Z} \sum_{A \in \alpha_{\hat \omega_1}} \int_{[0,1)} \phi\big(\lim_{n \rightarrow \infty} \mathbb{E}_{\hat{\mu}_{\hat \omega}}(1_A|\sigma(T_{\hat \omega_1}^{-1} \alpha_{\hat\tau \hat \omega,n}))\big)(x) \, d\hat{\mu}_{\hat \omega}(x) d\hat \nu (\hat \omega),
\end{split}\end{equation}
where for the last equality we also use the Monotone Convergence Theorem, the positivity of the integrands and the continuity of $\phi$. From Lemma~\ref{l:operator} with $\psi = 1_A$ we get from (i) that
\begin{equation}\label{q:lemma(i)}
\mathbb{E}_{\hat{\mu}_{\hat \omega}}(1_A|\sigma(T_{\hat \omega_1}^{-1} \alpha_{\hat\tau \hat \omega,n}))(x) =  \mathbb{E}_{\hat \mu_{\hat \tau \hat \omega}}(\mathcal L_{\hat \omega}1_A|\sigma(\alpha_{\hat \tau \hat \omega,n}))(T_{\hat \omega_1} x) \qquad \hat \mu_{ \hat \omega}\text{-a.e.}
\end{equation}
and from (ii) that
\begin{equation}\label{q:lemma(ii)}
\lim_{n \rightarrow \infty} \mathbb{E}_{\hat \mu_{\hat \tau \hat \omega}}(\mathcal L_{\hat \omega}1_A|\sigma(\alpha_{\hat \tau \hat \omega,n})) = \mathcal L_{\hat \omega}1_A \qquad \hat \mu_{\hat \tau \hat \omega}\text{-a.e.}
\end{equation}
Recall that $C_{\hat \omega} = \{ x \in [0,1) \, : \, \frac{d \hat \mu_{\hat \tau \hat \omega}}{d\lambda}(x)>0 \}$ satisfies $\hat \mu_{\hat \tau \hat \omega}(C_{\hat \omega})=1$. Combining \eqref{q:lemma(i)} and \eqref{q:lemma(ii)} with \eqref{q:hfibfirst} and using $\hat \mu_{\hat \tau \hat \omega} = \hat \mu_{\hat \omega} \circ T^{-1}_{\hat \omega_1}$ and a change of variables formula we conclude that
\[ \begin{split}
h^{\text{fib}}(\mathcal{T}) =\ & \int_{I^\mathbb Z} \sum_{A \in \alpha_{\hat \omega_1}} \int_{[0,1)} \phi\big(\mathcal L_{\hat \omega}1_A(x)\big)\, d\hat \mu_{\hat \tau \hat \omega}(x) \, d\hat \nu (\hat \omega)\\
=\ & - \int_{I^\mathbb Z} \sum_{A \in \alpha_{\hat \omega_1}} \bigg[ \int_{T_{\hat \omega_1}(A) \cap C_{\hat \omega}} \Big(\frac{1}{J T_{\hat \omega_1}} \frac{d\hat{\mu}_{\hat \omega}}{d\lambda} \log \Big( \frac{h_{\hat \omega} \circ T_{\hat \omega_1}}{J T_{\hat \omega_1}} \frac{d\hat{\mu}_{\hat \omega}}{d\lambda}\Big)\Big) \circ (T_{\hat \omega_1}|A)^{-1} (x)\, d\lambda (x)\bigg]d\hat \nu(\hat \omega) \\
=\ & - \int_{I^\mathbb Z} \sum_{A \in \alpha_{\hat \omega_1}} \int_{A \cap T_{\hat \omega_1}^{-1} C_{\hat \omega}}  \log \Big( \frac{h_{\hat \omega} \circ T_{\hat \omega_1}}{J T_{\hat \omega_1}} \frac{d\hat \mu_{\omega}}{d\lambda}\Big)(x) \, d\hat \mu_{\hat \omega}(x) d\hat \nu (\hat \omega) \\
=\ & \int_{I^\mathbb Z \times [0,1)} \log J T_{\hat \omega_1}(x) \, d\hat \mu (\hat \omega,x) - \int_{I^\mathbb Z \times [0,1)} \log h_{\hat \omega} \circ T_{\hat \omega_1}(x) \, d\hat \mu(\hat \omega,x)\\
& - \int_{I^\mathbb Z \times [0,1)} \log \frac{d{\hat \mu_{\hat \omega}}}{d\lambda}(x) \, d\hat \mu(\omega,x).
\end{split}\]
Set
\[ \eta(\hat \omega,x) = \begin{cases} 
\big( \frac{d\hat \mu_{\hat \omega}}{d\lambda}(x) \big)^{-1}, & \text{if } \frac{d{\hat \mu_{\hat \omega}}}{d\lambda}(x) > 0,\\
1, & \text{if } \frac{d{\hat \mu_{\hat \omega}}}{d\lambda}(x) = 0.
\end{cases} \]
Then for each $x \in T_{\hat \omega_1}^{-1}(C_{\hat \omega})$,
\[ h_{\hat \omega} \circ T_{\hat \omega_1}(x) = \eta \circ \hat F (\hat \omega,x).\]
With the $\hat F$-invariance of $\hat \mu$ this yields
\[ \int_{I^\mathbb Z \times [0,1)} \log h_{\hat \omega} \circ T_{\hat \omega_1}(x) \, d\hat \mu(\hat \omega,x) = \int_E \int_{T_{\hat{\omega}_1}^{-1} C_{\hat{\omega}}} \log  \eta \circ \hat F(\hat \omega,x) \, d\hat \mu_{\hat \omega}(x) d\hat \nu (\hat \omega) = \int_{I^\mathbb Z \times [0,1)} \log \eta \, d\hat \mu.\]
The result now follows.
\end{proof}

\begin{proof}[Proof of Theorem \ref{t:main}] 
Part (i) an (iii) are given by Theorem~\ref{prop2.6a} and Theorem~\ref{prop2.7a}, respectively. By \eqref{q:AR} for part (ii) it is enough to show that $h_{\hat \mu}(\hat F) = h_\mu(F)$ and $h_{\hat \nu}(\hat \tau) = h_\nu(\tau)$. The latter follows immediately, because $(I^\mathbb Z, \mathcal B_I^\mathbb Z, \hat \nu, \hat \tau)$ is the {\em natural extension} of the system $(I^\mathbb N, \mathcal B_I^\mathbb N, \nu, \tau)$, i.e., $(I^\mathbb Z, \mathcal B_I^\mathbb Z, \hat \nu, \hat \tau)$ is the smallest (in the sense of the $\sigma$-algebra) invertible system that has $(I^\mathbb N, \mathcal B_I^\mathbb N, \nu, \tau)$ as a factor, and entropy is preserved under this construction (see e.g.~\cite[Chapter 5]{kalle2}). For the first part, note that $(I^\mathbb N \times [0,1), \mathcal B_I^\mathbb N \times \mathcal B, \mu, F)$ is a factor of $(I^\mathbb Z \times [0,1), \mathcal B_I^\mathbb Z \times \mathcal B, \hat \mu, \hat F)$. Hence, $h_{\hat \mu}(\hat F) \ge h_\mu(F)$. Conversely, let $\mathcal{C} = \{C_1, C_2, \ldots\}$ be a countable set that generates $\mathcal B_I$, so $\sigma(\mathcal{C}) = \mathcal{B}_I$. Define for each $n \in \mathbb{N}$ the partition $\gamma_n$ of $I$ given by $\gamma_n = \{ C_i \backslash \bigcup_{j=1}^{i-1} C_j \, : \, i=1,\ldots,n\} \cup \{X \backslash \bigcup_{i=1}^n C_i  \}$. Then $\sigma(\gamma_n) \subseteq  \sigma(\gamma_{n+1})$ for each $n \ge 1$ and $\sigma(\lim_{n \rightarrow \infty} \gamma_n) = \mathcal{B}_I$. Similarly, since $\mathcal B$ is countably generated there exists a sequence of finite partitions $(\beta_n)_{n\ge 1}$ of $[0,1)$ such that $\sigma(\lim_{n \rightarrow \infty} \beta_n) = \mathcal B$. Then
\[ \xi_n= \{ (\cdots \times I \times A_{-n} \times A_{-n+1} \times \cdots \times A_n \times I \times \cdots) \times B: A_i \in \gamma_n, i=-n,\ldots,n, B \in \beta_n\}\]
(where $A_0$ is on the 0-th position) is a finite partition of $I^{\mathbb{Z}} \times [0,1)$ such that $\sigma(\lim_{n \rightarrow \infty} \xi_n) = \mathcal B_I^{\mathbb Z} \times \mathcal{B}$. Note that $\hat F^{-n-1}\xi_n$ specifies sets in positions 1 to $2n+1$ in the first coordinate. Recall that $\Pi: I^\mathbb Z \times [0,1) \to I^\mathbb N \times [0,1)$ denotes the canonical projection. We conclude using Lemma~\ref{l:increasingentropy} that
\[ \begin{split}
h_{\hat \mu}(\hat F) =\ & \lim_{n \rightarrow \infty} h_{\hat \mu}(\xi_n , \hat F) = \lim_{n \rightarrow \infty} h_{\hat \mu}(\hat F^{-n-1} \xi_n, \hat F) \\
=\ & \lim_{n \rightarrow \infty} h_{\mu}(\Pi(\hat F^{-n-1} \xi_n), F) \leq h_{\mu}(F).
\end{split}\]
This finishes the proof.
\end{proof}

\subsection{The Central Limit Theorem}
In case we compare a random number system to an NTFM we can obtain a Central Limit Theorem for Theorem~\ref{t:randomlochs} in a way comparable to what has been done in \cite{atilla} for two NTFM's. Given a random number system $\mathcal{T}=(I, \mathbf P, \{ T_i \}_{i \in I},\mu, \{ \alpha_i\}_{i \in I})$ and an NTFM $(S, \tilde \mu,  \mathcal E)$, for each $n \in \mathbb{N}$, $\omega \in I^{\mathbb N}$ and $x \in [0,1)$, let
\[ m_{\mathcal T,S}(n,\omega,x) = \sup\{m \in \mathbb{N}: \alpha_{\omega,n}(x) \subseteq \mathcal E_m(x)\}.\]
This is the analog of \eqref{eq:1.15} for comparing two random number systems. We first introduce additional assumptions that we put on the systems. The following property can be found in \cite[Property 2.1]{atilla}.

\begin{defn}\label{d:zero}
Let $(S, \tilde \mu,  \mathcal E)$ be an NTFM. We say that $S$ satisfies the {\em zero-property} if
\[ \lim_{n \rightarrow \infty} \frac{-\log \lambda\big(\mathcal E_n(x)\big) - n h_{\tilde \mu}(S)}{\sqrt{n}} = 0 \qquad \lambda\text{-a.e.}\]
\end{defn}

The zero-property is rather strong, but \cite[Section 3.2]{atilla} contains several examples of NTFM's that satisfy it, including the maps $S(x) = M x \pmod 1$ for integers $M \ge 2$ and $\beta$-transformations $S(x) = \beta x \pmod 1$ with $\beta>1$ a so-called {\em Parry number}, i.e., a number $\beta$ for which the set $\{ S^n(\beta-1) \, : \, n \ge 0 \}$ is finite. The following lemma, which compares to \cite[Lemma 2.1]{atilla}, yields a consequence of the zero-property.
\begin{lemma}[cf.~Lemma 2.1 in \cite{atilla}]\label{lemma0.7}
Let $\mathcal T = (I, \mathbf P, \{T_i\}_{i \in I}, \mu, \{\alpha_i\}_{i \in I})$ be a random number system and let $(S, \tilde \mu, \mathcal E)$ be an NTFM that satisfies the zero-property. Assume that $h^{\rm{fib}}(\mathcal T), h_{\tilde \mu}(S) \in (0,\infty)$. Then
\[ \log\bigg(\frac{\lambda(\mathcal E_{m_{\mathcal T,S}(n,\omega,x)}(x))}{\lambda(\alpha_{\omega,n}(x))}\bigg) = o(\sqrt n) \quad \text{in $\mu$-probability}. \]
That is, for each $\varepsilon > 0$,
\[ \lim_{n \to \infty} \mu\big(\big\{(\omega,x): \big| \log \lambda(\mathcal E_{m_{\mathcal T,S}(n,\omega,x)}(x)) - \log \lambda(\alpha_{\omega,n}(x)) \big| > \varepsilon \sqrt{n} \big\} \big) =0. \]
\end{lemma}

\begin{proof} By definition of $m_{\mathcal T,S}(n,\omega,x)$ we see that $\alpha_{\omega,n}(x) \subseteq \mathcal E_{m_{\mathcal T,S}(n,\omega,x)}(x)$. Since $\mu$ and $\mathbf P \times \lambda$ are equivalent, it suffices to show that for all $\varepsilon, \tilde{\varepsilon} > 0$ there exists an $N \in \mathbb N$ such that for all $n > N$ we have
\[
\mathbf P \times \lambda\big(\big\{(\omega,x)\, :\, \log \lambda(\mathcal E_{m_{\mathcal T,S}(n,\omega,x)}(x)) - \log \lambda(\alpha_{\omega,n}(x)) > \varepsilon \sqrt n \big\} \big) < \tilde{\varepsilon}.
\]
Fix $\varepsilon, \tilde{\varepsilon} > 0$ and set $\eta = \varepsilon \sqrt{\frac{h_{\tilde \mu}(S)}{3h^{\text{fib}}(\mathcal T)}}$. For each $n \in \mathbb{N}$, we put
\[ A_n = I^\mathbb N \times \big\{ x \in [0,1]\, :\,  \exists \, k \geq n \, \text{ s.t. } | \log \lambda\big(\mathcal E_k(x)\big) + k h_{\tilde \mu}(S) | > \frac{1}{2} \eta \sqrt k\, \big\}.\]
Because $S$ satisfies the zero-property, we know that there exists an $n_0 \in \mathbb N$ such that $\mathbf P \times \lambda (A_{n_0} ) \le \frac{\tilde{\varepsilon}}{3}$. Note that if $(\omega,x) \notin A_{n_0}$, then for each $n > n_0$ we have
\[ \begin{split}
\log \lambda(\mathcal E_{n-1}(x)) \le \ &  \frac12 \eta \sqrt{n} - (n-1)h_{\tilde \mu}(S),\\
\log \lambda(\mathcal E_n(x)) \ge \ & - \frac12 \eta \sqrt{n} -  nh_{\tilde \mu}(S),
\end{split}\]
which when combined leads to
\begin{equation}\label{eq:10}
\log \lambda\big(\mathcal E_{n-1}(x)\big) - \log \lambda(\mathcal E_n(x)\big) \le h_{\tilde \mu}(S) + \eta \sqrt n.
\end{equation}
For each $n \ge 1$ put
\begin{align}\label{eq:11}
B_n = \Big\{ (\omega,x) \in I^\mathbb N \times [0,1) \, : \, \frac{1}{2} < \frac{m_{\mathcal T,S}(k,\omega,x)}{k} \frac{h_{\tilde \mu}(S)}{h^{\text{fib}}(\mathcal T)} < 2  \quad  \forall \, k \geq n \Big\}.
\end{align}
From Theorem \ref{t:randomlochs} it follows that there exists an integer $n_1 > 2n_0 \frac{h_{\tilde \mu}(S)}{h^{\text{fib}}(\mathcal T)}$ such that $\mathbf P \times \lambda(B_{n_1}) > 1 - \frac{\tilde{\varepsilon}}{3}$. Note from \eqref{eq:11} that for each $(\omega,x) \in B_{n_1}$ we have $n_0 < \frac{1}{2} n_1 \frac{h^{\text{fib}}(\mathcal T)}{h_{\tilde \mu}(S)} < m_{\mathcal T,S}(n_1,\omega,x)$. Therefore, it follows from \eqref{eq:10} that, for each $(\omega,x) \in B_{n_1} \setminus A_{n_0}$ and for each $n > n_1$,
\begin{align}\label{eq:12}
\log \lambda\big(\mathcal E_{m_{\mathcal T,S}(n,\omega,x)}(x)\big) \le h_{\tilde \mu}(S) + \eta \sqrt{m_{\mathcal T,S}(n,\omega,x)+1} + \log \lambda\big(\mathcal E_{m_{\mathcal T,S}(n,\omega,x)+1}(x)\big).
\end{align}
For each $n \ge 1$ and interval $E \in \mathcal E_n$ use $\partial E$ to denote the boundary of $E$, i.e., the collection of its two endpoints, and use $\text{dist}(x,\partial E) = \inf \{ |x-a| \, : \, a \in \partial E \}$ to denote the distance from $x$ to the nearest boundary point of $E$. For each $n \in \mathbb N$ and $(\omega,x) \in I^\mathbb N \times [0,1)$ we have that $\alpha_{\omega,n}(x) \nsubseteq \mathcal E_{m_{\mathcal T,S}(n,\omega,x)+1}(x)$, so $\text{dist}(x, \partial \mathcal E_{m_{\mathcal T,S}(n,\omega,x)+1}(x)) \le \lambda(\alpha_{\omega,n}(x))$ and with \eqref{eq:12} we obtain for each $(\omega,x) \in B_{n_1} \setminus A_{n_0}$ and each $n > n_1$ that
\begin{multline}\label{eq:14}
\frac{\log \lambda\big(\mathcal E_{m_{\mathcal T, S}(n,\omega,x)}(x)\big) - \log \lambda\big(\alpha_{\omega,n}(x)\big)}{\sqrt{n}} \\
\le \frac{\log \lambda\big(\mathcal E_{m_{\mathcal T, S}(n,\omega,x)+1}(x)\big) - \log \text{dist}(x, \partial \mathcal E_{m_{\mathcal T,S}(n,\omega,x)+1}(x))}{\sqrt{n}}   + \frac{h_{\tilde \mu}(S)}{\sqrt{n}} +  \eta \sqrt{\frac{m_{\mathcal T, S}(n,\omega,x)+1}{n}}.
\end{multline}
For each $n \in \mathbb{N}$ and interval $E \in \mathcal E_n$, we define a new interval $E'$ by removing from both ends of $E$ an interval of length $\frac{\tilde{\varepsilon}}{\pi^2 n^2}\lambda(E)$. Furthermore, we define
\begin{align*}
C = [0,1) \backslash \Big( \bigcup_{n \in \mathbb{N}} \bigcup_{E \in \mathcal{E}_n} E \backslash E'\Big).
\end{align*}
Then
\begin{align*}
\lambda(C) \geq 1 - \sum_{n \in \mathbb{N}} \sum_{E \in \mathcal{E}_n} \frac{2\tilde{\varepsilon}}{\pi^2 n^2}\lambda(E) = 1- \frac{\tilde{\varepsilon}}{3}, \\
\end{align*}
so $\mathbf{P} \times \lambda(\big(B_{n_1} \cap (I^\mathbb N \times C) \big) \setminus A_{n_0}\big) \geq 1 - \tilde{\varepsilon}$. For each $(\omega,x) \in I^{\mathbb{N}} \times C$ and $n \in \mathbb{N}$ we have the bound
\[ \text{dist}( x, \partial \mathcal E_{m_{\mathcal T, S}(n,\omega,x)+1}(x)) \ge \frac{\tilde{\varepsilon}}{\pi^2 n^2}\lambda(\mathcal E_{m_{\mathcal T, S}(n,\omega,x)+1}(x)).\]
Combining this with \eqref{eq:11} and \eqref{eq:14} gives for each integer $n > n_1$ and each $(\omega,x) \in \big(B_{n_1} \cap (I^\mathbb N \times C) \big) \setminus A_{n_0}$ that
\begin{align}\label{eq:15}
\frac{\log \lambda\big(\mathcal E_{m_{\mathcal T, S}(n,\omega,x)}(x)\big) - \log \lambda\big(\alpha_{\omega,n}(x)\big)}{\sqrt{n}} \leq \frac{|\log \frac{\pi^2 n^2}{\tilde{\varepsilon}}| + h_{\tilde \mu}(S)}{\sqrt{n}} + \eta \sqrt{\frac{2h^{\text{fib}}(\mathcal T)}{h_{\tilde \mu}(S)} + \frac{1}{n}}.
\end{align}
We now take an integer $N > n_1$ large enough such that for each integer $n \geq N$ the right-hand side of \eqref{eq:15} is bounded by $\eta \sqrt{\frac{3h^{\text{fib}}(\mathcal T)}{h_{\tilde \mu}(S)}} = \varepsilon$. This yields the lemma.
\end{proof}

Furthermore, we consider random number systems with the following property.

\begin{defn}\label{d:clt}
Let $\mathcal T = (I, \mathbf P, \{T_i\}_{i \in I},\mu, \{ \alpha_i\}_{i \in I})$ be a random number system. We say that $\mathcal T$ satisfies the {\em CLT-property} if there is some $\sigma > 0$ such that
\[ \lim_{n \rightarrow \infty} \mu\Big( \Big\{ (\omega,x) \in I^\mathbb N \times [0,1) \, : \, \frac{-\log \lambda(\alpha_{\omega,n}(x)) - n h^{\text{fib}}(\mathcal T)}{\sigma \sqrt{n}} \leq u \Big\} \Big) = \frac{1}{\sqrt{2\pi}}\int_{-\infty}^u e^{-t^2 / 2} dt\]
holds for each $u \in \mathbb{R}$.
\end{defn}

With the properties from Definitions~\ref{d:zero} and \ref{d:clt} we obtain the following result.
\begin{theorem}\label{t:clt}
Let $\mathcal T = (I, \mathbf P, \{T_i\}_{i \in I}, \mu, \{ \alpha_i\}_{i \in I})$ be a random number system satisfying the CLT-property with constant $\sigma>0$ and let $(S, \tilde \mu, \mathcal E)$ be an NTFM that satisfies the zero-property. Assume that $h^{\mathrm{fib}}(\mathcal T), h_{\tilde \mu}(S) \in (0, \infty)$. Then for all $u \in \mathbb{R}$,
\[ \lim_{n \rightarrow \infty} \mu\Big(\Big\{(\omega,x) \in I^\mathbb N \times [0,1)\, :\,  \frac{m_{\mathcal T,S}(n,\omega,x) - n\frac{h^{\text{fib}}(\mathcal T)}{h_{\tilde \mu}(S)}}{\kappa \sqrt n} \leq u\Big\}\Big) = \frac1{\sqrt{2 \pi}} \int_{-\infty}^u e^{-t^2/2} \, dt,\]
where $\kappa = \frac{\sigma}{h_{\tilde \mu}(S)}$.
\end{theorem}

\begin{proof} 
Fix some $u \in \mathbb R$. We rewrite
\begin{equation}\begin{split}\label{eq:41}
\frac{m_{\mathcal T,S}(n,\omega,x)-n \frac{h^{\text{fib}}(\mathcal T)}{h_{\tilde \mu}(S)}}{\kappa \sqrt n} =\ & \frac{-\log \lambda(\alpha_{\omega,n}(x)) - n h^{\text{fib}}(\mathcal T)}{h_{\tilde \mu}(S)\kappa \sqrt n}\\
& + \frac{\log \lambda(\alpha_{\omega,n}(x)) - \log \lambda(\mathcal E_{m_{\mathcal T,S}(n,\omega,x)}(x))}{h_{\tilde \mu}(S) \kappa \sqrt n} \\
 & + \frac{\log \lambda(\mathcal E_{m_{\mathcal T,S}(n,\omega,x)}(x)) + h_{\tilde \mu}(S) m_{\mathcal T,S}(n,\omega,x) }{h_{\tilde \mu}(S) \kappa \sqrt n}.
\end{split}\end{equation}
The last term can be written as
\begin{equation}\begin{split}\label{eq:42}
\frac{\log \lambda(\mathcal E_{m_{\mathcal T,S}(n,\omega,x)}(x))+ h_{\tilde \mu}(S) m_{\mathcal T,S}(n,\omega,x) }{h_{\tilde \mu}(S) \kappa \sqrt n} &\\
= \frac{1}{h_{\tilde \mu}(S) \kappa} \sqrt{\frac{m_{\mathcal T,S}(n,\omega,x)}{n}} & \frac{\log \lambda(\mathcal E_{m_{\mathcal T,S}(n,\omega,x)}(x)) + h_{\tilde \mu}(S) m_{\mathcal T,S}(n,\omega,x)}{\sqrt{m_{\mathcal T,S}(n,\omega,x)}}.
\end{split}\end{equation}
From Theorem~\ref{t:randomlochs} we know that $\lim_{n \to \infty} \frac{m_{\mathcal T,S}(n,\omega,x)}{n} = \frac{h^{\text{fib}}(\mathcal T)}{h_{\tilde \mu}(S)} < \infty$ for $\mathbf P \times \lambda$-a.e.~$(\omega,x) \in I^\mathbb N \times [0,1)$. Since $\lim_{n \to \infty} m_{\mathcal T,S}(n,\omega,x) = \infty$ for $\mathbf P \times \lambda$-a.e.~$(\omega,x) \in I^\mathbb N \times [0,1)$, it then follows from the zero-property of $S$ that \eqref{eq:42} converges to 0 as $n \rightarrow \infty$ $\mathbf P \times \lambda$-a.e.~and thus also $\mu$-a.e. Hence it converges in $\mu$-probability as well. Furthermore, we know from Lemma~\ref{lemma0.7} that the second term on the right-hand side of \eqref{eq:41} converges to 0 as $n \rightarrow \infty$ in $\mu$-probability. 

\vskip .2cm
Define three sequences of random variables $(X_n)_{n \ge 1}, (Y_n)_{n \ge 1}$ and $(Z_n)_{n \ge 1}$ on $I^\mathbb N \times [0,1)$ by setting
\[ \begin{split}
X_n=\ & \frac{m_{\mathcal T,S}(n,\omega,x)-n \frac{h^{\text{fib}}(\mathcal T)}{h_{\tilde \mu}(S)}}{\kappa \sqrt n},\\
Y_n= \ & \frac{-\log \lambda(\alpha_{\omega,n}(x)) - n h^{\text{fib}}(\mathcal T)}{h_{\tilde \mu}(S)\kappa \sqrt n},\\
Z_n= \ & \frac{\log \lambda(\alpha_{\omega,n}(x)) - \log \lambda(\mathcal E_{m_{\mathcal T,S}(n,\omega,x)}(x))}{h_{\tilde \mu}(S) \kappa \sqrt n} \\
 & + \frac{\log \lambda(\mathcal E_{m_{\mathcal T,S}(n,\omega,x)}(x)) + h_{\tilde \mu}(S) m_{\mathcal T,S}(n,\omega,x) }{h_{\tilde \mu}(S) \kappa \sqrt n}.
\end{split}\]
Then by the above for each $\varepsilon >0$, $\lim_{n \to \infty} \mu(|Z_n| > \varepsilon )=0$ and since $\mathcal T$ satisfies the CLT-property for each $u \in \mathbb R$ it holds that
\[ \lim_{n \to \infty} \mu(Y_n \le u) = \frac{1}{\sqrt{2\pi}}\int_{-\infty}^u e^{-t^2 / 2} dt.\]
Fix some $u \in \mathbb R$ and some $\varepsilon >0$. We are interested in $\lim_{n \to \infty} \mu(X_n \le u)$. From \eqref{eq:41} we see that
\[ \begin{split}
\mu(X_n \le u) = \mu(Y_n \le u-Z_n) = \mu(Y_n \le u-Z_n, \, |Z_n| \le \varepsilon) + \mu(Y_n \le u-Z_n, \, |Z_n|> \varepsilon).
\end{split}\]
Since $\lim_{n \to \infty} \mu(Y_n \le u-Z_n, \, |Z_n|> \varepsilon)=0$ we get
\[ \begin{split}
\limsup_{n \to \infty} \mu(X_n \le u) =\ & \limsup_{n \to \infty}  \mu(Y_n \le u-Z_n, \, |Z_n| \le \varepsilon)\\
\le \ & \limsup_{n \to \infty} \mu(Y_n \le u+\varepsilon)\\
= \ & \frac{1}{\sqrt{2\pi}}\int_{-\infty}^{u+\varepsilon} e^{-t^2 / 2} dt.
\end{split}\]
On the other hand,
\[ \begin{split}
\liminf_{n \to \infty} \mu(X_n \le u) =\ & \liminf_{n \to \infty}  \mu(Y_n \le u-Z_n, \, |Z_n| \le \varepsilon)\\
\ge \  & \liminf_{n \to \infty} \mu(Y_n \le u-\varepsilon,  \, |Z_n| \le \varepsilon)\\
\ge \ & \liminf_{n \to \infty} \big( \mu(Y_n \le u-\varepsilon) - \mu( |Z_n| > \varepsilon) \big)\\
= \ & \frac{1}{\sqrt{2\pi}}\int_{-\infty}^{u-\varepsilon} e^{-t^2 / 2} dt.
\end{split}\]
Since this holds for all $\varepsilon>0$ we get the result.
\end{proof} 

We will now identify a class of random number systems that satisfies the CLT property. For this, recall that a $C^3$ map $T: J \rightarrow J$ on an interval $J$ is said to have \emph{non-positive Schwarzian derivative} if for each $x \in J$ we have $DT(x) \neq 0$ and the \emph{Schwarzian derivative} of $T$ at $x$ given by
\[ \mathbf{S}T(x) = \frac{D^3 T(x)}{DT(x)}-\frac{3}{2} \Big(\frac{D^2 T(x)}{DT(x)}\Big)^2\]
satisfies $\mathbf{S}T(x) \leq 0$. A consequence of $\mathbf{S}T(x) \leq 0$ for all $x \in J$ is that $T$ is monotone on $J$. Furthermore, an easy calculation shows that the Schwarzian derivative of the composition of two transformations $T_1,T_2: J \rightarrow J$ satisfies $\mathbf{S}(T_2 \circ T_1) \leq 0$ provided $\mathbf{S}T_1 \leq 0$ and $\mathbf{S}T_2 \leq 0$.

\vskip .2cm
Now, let $\{T_i: [0,1) \rightarrow [0,1)\}_{i \in I}$ be a countable collection of transformations that satisfy the following two properties:
\vskip .1cm
(p1) For each $i \in I$ there exists an interval partition $\alpha_i $ of $[0,1)$, such that for each $A \in \alpha_i$, $T_i|_A$ has non-positive Schwarzian derivative and $T_i(A) = [0,1)$.
\vskip .1cm
(p2) There exist $1 < K \le M < \infty$ such that, for all $x \in [0,1)$ and all $i \in I$,
\begin{align}\label{eq36}
K \leq |DT_i(x)| \leq M.
\end{align}
In particular, each $T_i$ is expanding and has finitely many onto branches (so each partition $\alpha_i$ has at most finitely many non-empty intervals) and $T_i|_A$ is $C^3$ for each $A \in \alpha_i$. We let $\mathcal D$ denote the set of all collections $\{ T_i \}_{i \in I}$ satisfying conditions (p1)~and (p2)~for some countable index set $I$ and show that each element of $\mathcal D$ gives a random number system, that under some additional assumptions satisfies the CLT-property.

\begin{prop}\label{prop3.1}
Let $\{ T_i \}_{i \in I} \in \mathcal D$, $\{ \alpha_i \}_{i \in I}$ the set of partitions given by Property (p1)~and $\mathbf P$ be a Bernoulli measure on $(I^\mathbb N, \mathcal B_I^\mathbb N)$. Then there exists a unique measure $\mu$ such that $\mathcal{T} = (I,\mathbf{P},\{T_i\}_{i \in I}, \mu, \{\alpha_i\}_{i \in I})$ is a random number system. Moreover, $\mathcal T$ satisfies the CLT-property if and only if for each measurable function $\psi: I^{\mathbb{N}} \times [0,1) \rightarrow \mathbb{R}$ we have $\varphi \circ F \neq \psi \circ F - \psi$, where $F$ is the skew product associated to $\{T_i\}_{i \in I}$ and $\varphi: I^{\mathbb{N}} \times [0,1) \rightarrow \mathbb{R}$ is given by
\begin{align}
\varphi(\omega,x) = \log |DT_{\omega_1}(x)| -  h^{\mathrm{fib}}(\mathcal{T}).
\end{align}
\end{prop}

To prove Proposition~\ref{prop3.1} we use two theorems by Young \cite{Y99}. The results from \cite{Y99} are formulated for Young towers, i.e., extensions of induced systems for suitable return time maps. We will, however, apply them to the system itself. That is, we will take the whole space $I^\mathbb N \times [0,1)$ as the inducing domain and as a consequence the return time function $R: I^\mathbb N \times [0,1) \to \mathbb N$ will have $R(\omega,x)=1$ for each $(\omega,x)$. In particular $\int_{I^\mathbb N \times [0,1)} R \, d \mathbf P \times \lambda =1$. For the convenience of the reader, we will reformulate the results from \cite{Y99} that are relevant for the proof of Proposition~\ref{prop3.1} here for our setting, together with the necessary conditions.

\vskip .2cm
The skew product $F$ maps each element $[i] \times A$, $i \in I$ and $A \in \alpha_i$, bijectively onto $I^\mathbb N \times [0,1)$. Moreover, both $F|_{[i] \times A}$ and its inverse are non-singular with respect to $\mathbf P \times \lambda$ (giving (r4)). Hence, the Jacobian $JF$ with respect to $\mathbf P \times \lambda$ exists and is positive $\mathbf P \times \lambda$-a.e. By condition (p2)~the collection $\{ [i] \times A \, : \, i \in I, \, A \in \alpha_i\}$ generates the $\sigma$-algebra $\mathcal B_I^\mathbb N \times \mathcal B$ (giving (r5)). For each $(\omega,x), (\tilde \omega,y) \in I^\mathbb N \times [0,1)$ write $s\big( (\omega,x), (\tilde \omega,y)\big)$ for the smallest $n \geq 0$ such that $F^n(\omega,x)$ and $F^n(\tilde \omega,y)$ lie in distinct sets $[i] \times A$. The results from \cite[Theorem 1]{Y99} then imply, among other things, the following: if there are $C_1>0$, $\eta \in (0,1)$ such that for all $[i] \times A$ and all $(\omega,x), (\tilde \omega,y) \in [i] \times A$ it holds that
\begin{equation}\label{q:young1}
\left| \frac{JF(\omega,x)}{JF(\tilde \omega,y)} -1 \right| \le C_1\eta^{s(F(\omega,x), F(\tilde \omega,y))},
\end{equation}
then $F$ admits an invariant and ergodic invariant probability measure $\mu$ that is absolutely continuous with respect to $\mathbf P \times \lambda$ with a density that is bounded away from 0. We will use this to show that each $\{ T_i \}_{i \in I} \in \mathcal D$ yields a random number system.

\vskip .2cm
For the statement about the CLT-property in Proposition~\ref{prop3.1} we apply \cite[Theorem 4]{Y99} to the function $\varphi$. For this we need to verify that $\int_{I^\mathbb N \times [0,1)} \varphi \, d\mu =0$ and that there is a constant $C_2>0$ such that
\begin{equation}\label{q:young2}
|\varphi(\omega,x)-\varphi(\tilde \omega,y)| \le C_2 \eta^{s((\omega,x),(\tilde \omega,y))}
\end{equation}
for all $(\omega,x), (\tilde \omega,y) \in I^\mathbb N \times [0,1)$, where $\eta$ is the constant from \eqref{q:young1}. Under these conditions \cite[Theorem 4]{Y99} states that the sequence $\big(\frac1{\sqrt n} \sum_{i=0}^{n-1} \varphi \circ F^i \big)_n$ converges to a normal distribution with mean $0$ and variance $\sigma^2$ for some $\sigma>0$ in distribution with respect to $\mu$ if and only if $\varphi \circ F \neq \psi \circ F - \psi$ for any measurable function $\psi: I^{\mathbb{N}} \times [0,1) \to \mathbb R$.

\begin{proof}
Let $\{ T_i \}_{i \in I} \in \mathcal D$, $\{ \alpha_i \}_{i \in I}$ the set of partitions from Property (p1)~and $\mathbf{P}$ a Bernoulli measure on $(I^\mathbb N, \mathcal B_I^\mathbb N)$ with probability vector $(p_i)_{i \in I}$. A suitable invariant measure $\mu$ for the skew product $F$ is obtained from \cite[Theorem 1]{Y99} once we show that \eqref{q:young1} holds. First note that (r1), (r2), (r3) follow straightforwardly and (r4), (r5) were already addressed above. By Property (p2)~the partition $\Delta$ is finite, yielding (r7) once we have $\mu$. Hence, we focus on (r6).

\vskip .2cm
Since each branch of each $T_i$ has non-positive Schwarzian derivative and $\inf_{(i,x)} |DT_i(x)| > 1$, it follows from the Koebe Principle (see e.g.~\cite[Section 4.1]{dMvS93}) that there exists $K_1, K_2 > 0$ such that for each $\omega \in I^{\mathbb{N}}$, $n \in \mathbb{N}$, $A \in \alpha_{\omega,n}$ and $x,y \in A$ we have
\begin{align}\label{q:koebe1}
\frac{1}{K_1} \leq \frac{DT_{\omega}^n(x)}{DT_{\omega}^n(y)} \leq K_1
\end{align}
and
\begin{align}\label{eq39}
\Big| \frac{DT_{\omega}^n(x)}{DT_{\omega}^n(y)}-1\Big| \leq K_2 \cdot \frac{|T_{\omega}^n(x)-T_{\omega}^n(y)|}{\lambda(T_{\omega}^n(A))} =  K_2 \cdot |T_{\omega}^n(x)-T_{\omega}^n(y)|.
\end{align}
For each $i\in I$ and $A \in \alpha_i$ we have for all measurable sets $E \subseteq [i]$ and $B \subseteq A$ that
\[ \mathbf{P} \times \lambda\big(F(E \times B)) = \int_{E \times B} \frac{1}{p_{\omega_1}} DT_{\omega_1}(x) \, d\mathbf{P} \times \lambda(\omega,x),\]
from which it follows that
\begin{align}\label{eq44}
JF(\omega,x) = \frac{1}{p_{\omega_1}} DT_{\omega_1}(x).
\end{align}
Combining this with \eqref{eq39} yields for $i$ and $A$ and $(\omega,x),(\tilde{\omega},y) \in [i] \times A$ that
\begin{align}\label{eq45}
\Big|\frac{JF(\omega,x)}{JF(\tilde{\omega},y)}-1\Big| = \Big| \frac{DT_i(x)}{DT_i(y)}-1\Big| \leq K_2 \cdot |T_i(x)-T_i(y)|.
\end{align}
Assume $s\big( F(\omega,x), F(\tilde \omega,y)) =n$. Then for each $2 \le k \le n+1$ we have $\omega_k = \tilde \omega_k$ and that $T_\omega^{k-1} (x)$ and $T_{\tilde \omega}^{k-1}(y) = T_\omega^{k-1}(y)$ are in the same interval of the partition $\alpha_{\omega_{k+1}}$. It thus follows from the Mean Value Theorem and property (p2)~that 
\[ K \le \min |DT_{\omega_{k+1}}| \leq \Big|\frac{T_{\omega_{k+1}}(T_\omega^k(x)) - T_{\omega_{k+1}}(T_\omega^k (y))}{T_\omega^k(x) - T_\omega^k (y)}\Big| = \Big|\frac{T_\omega^{k+1}(x) - T_\omega^{k+1}(y)}{T_\omega^k (x) - T_\omega^k (y)}\Big|.\]
We conclude that
\[ |T_i(x)-T_i(y)| \le K^{-1} |T_\omega^2 (x) - T_\omega^2 (y)| \le \cdots \le K^{-n} |T_\omega^{n+1} (x) - T_\omega^{n+1} (y)| \le K^{-n}.\]
Together with \eqref{eq45} this shows that \eqref{q:young1} holds. Hence, we obtain an $F$-invariant and ergodic measure $\mu$ that is equivalent to $\mathbf P \times \lambda$. This implies that $\mathcal{T} = (I,\mathbf{P},\{T_i\}_{i \in I}, \mu, \{\alpha_i\}_{i \in I})$ is a random number system.

\vskip .2cm
What is left is to prove the statement on the CLT-property. Note that
\[ \frac1{\sqrt n} \sum_{i=0}^{n-1} \varphi \circ F^i(\omega,x) = \frac{ \log |DT^n_\omega (x)| - n h^{\mathrm{fib}}(\mathcal T) }{\sqrt n}.\]
Since $\mathcal{T} = (I,\mathbf{P},\{T_i\}_{i \in I}, \mu, \{\alpha_i\}_{i \in I})$ is a random number system, Theorem~\ref{t:main}(iii) implies that $\int_{I^\mathbb N \times [0,1)} \varphi \, d\mu =0$. Assume for a moment that also condition \eqref{q:young2} holds, i.e., that we satisfy the conditions of \cite[Theorem 4]{Y99}. It then follows that there is some $\sigma >0$ such that
\[ \lim_{n \to \infty} \mu \Big( \Big\{ (\omega,x) \in I^\mathbb N \times [0,1) \, : \, \frac{\log|D T_\omega^n (x)| -nh^{\mathrm{fib}}(\mathcal T)}{\sigma \sqrt n} \le u \Big\} \Big) = \frac1{\sqrt{2\pi}} \int_{-\infty}^u e^{-t^2/2} \, dt\]
if and only if $\varphi \circ F \neq \psi \circ F - \psi$ for any measurable function $\psi: I^{\mathbb{N}} \times [0,1) \to \mathbb R$. From \eqref{q:koebe1} it follows that for each $\omega \in I^{\mathbb{N}}$, $n \in \mathbb{N}$ and $x \in [0,1)$,
\begin{align*}
\lambda\big(\alpha_{\omega,n}(x)\big) \leq \frac{1}{\inf_{y \in \alpha_{\omega,n}(x)}|DT_{\omega}^n(y)|} = \frac{|DT_{\omega}^n(x)|}{\inf_{y \in \alpha_{\omega,n}(x)}|DT_{\omega}^n(y)|} \cdot |DT_{\omega}^n(x)|^{-1} \leq K_1 \cdot |DT_{\omega}^n(x)|^{-1}
\end{align*}
and similarly
\begin{align*}
\lambda\big(\alpha_{\omega,n}(x)\big) \geq \frac{1}{\sup_{y \in \alpha_{\omega,n}(x)}|DT_{\omega}^n(y)|} \geq \frac{1}{K_1} \cdot |DT_{\omega}^n(x)|^{-1}.
\end{align*}
Hence, if for each $n \ge 1$ we write
\[ X_n (\omega,x) =  \frac{-\log \lambda(\alpha_{\omega,n}(x))-nh^{\mathrm{fib}}(\mathcal T)}{\sigma \sqrt n},\]
then
\[ \frac{-\log K_1}{\sigma \sqrt n} + \frac{\log|D T_\omega^n (x)| -nh^{\mathrm{fib}}(\mathcal T)}{\sigma \sqrt n} \le X_n (\omega,x) \le  \frac{\log K_1}{\sigma \sqrt n} + \frac{\log|D T_\omega^n (x)| -nh^{\mathrm{fib}}(\mathcal T)}{\sigma \sqrt n},\]
and we see that to prove the last part of the proposition, it is enough to show that \eqref{q:young2} holds.

\vskip .2cm
Let $(\omega,x), (\tilde \omega,y) \in I^\mathbb N \times [0,1)$. We first consider the case that $s\big( (\omega,x), (\tilde \omega,y) \big)>0$. Let $i \in I$ and $A \in \alpha_i$ be such that $(\omega,x), (\tilde \omega,y) \in [i] \times A$. It then follows from \eqref{eq44} that
\[ |\varphi(\omega,x)-\varphi (\tilde \omega,y)| = \Big| \log \Big| \frac{DT_{\omega_1} (x)}{D T_{\tilde \omega_1}(y)} \Big| \Big| = \Big| \log \Big| \frac{ JF(\omega,x)}{JF(\tilde \omega,y)} \Big| \Big|.\]
Recall from the first part of the proof that
\[ \Big|\frac{JF(\omega,x)}{JF(\tilde \omega,y)}-1\Big| \leq K_2 \cdot K^{-s(F(\omega,x),F(\tilde \omega,y))}.\]
Combining this with the fact that for all $x > 0$,
\[ |\log x| \leq \max \{x-1,x^{-1} -1\} \leq \max \{ |x-1|, |x^{-1} -1|\},\]
yields that
\[ \Big|\log \Big|\frac{JF(\omega,x)}{JF(\tilde \omega,y)}\Big|\Big| \leq K_2 \cdot K^{-s(F(\omega,x),F(\tilde \omega,y))} \leq \tilde K_2 \cdot K^{-s( (\omega,x), (\tilde \omega,y))},\]
where $\tilde K_2 = K_2 \cdot K$. In case $s\big( (\omega,x), (\tilde \omega,y) \big)=0$ we just notice from \eqref{eq36} it follows that
\[ \Big|\log \Big|\frac{DT_{\omega_1}(x)}{DT_{\tilde \omega_1} (y)}\Big|\Big| \le \log M - \log K.\]
By taking $C_2 = \max\{\tilde K_2, \log M - \log K\}$ we obtain the result.
\end{proof}

\begin{remark}{\rm {(i) A natural question is whether a central limit theorem would hold when comparing two random number systems. For two random number systems $\mathcal{T} = (I, \mathbf P,  \{ T_i  \}_{i \in I}, \mu, \{ \alpha_i\}_{i \in I})$ and $\mathcal{S} = (J, \mathbf Q, \{S_j\}_{j \in J}, \rho, \{\gamma_j \}_{j \in J})$ the limit statement for an annealed\footnote{An \emph{annealed} result describes properties of the system averaged over the base space. On the other hand, a \emph{quenched} result characterises the behaviour of the system for almost all fixed realisations.} result with respect to $\mathcal{S}$ would describe a subset of $I^\mathbb N \times J^\mathbb N \times [0,1)$. One might expect that a central limit theorem with measure $\mathbf P \times \mathbf Q \times \lambda$ holds for random number systems $\mathcal T$ and $\mathcal S$ with invariant measures $\mu = \mathbf P \times \lambda$ and $\rho = \mathbf Q \times \lambda$, respectively. For a quenched result with respect to $\mathcal{S}$ the arguments for Theorem \ref{t:clt} might work if $\mathcal S$ satisfies a random zero-property, where in Definition \ref{d:zero} we replace $\mathcal{E}_n(x)$ and $h_{\tilde{\mu}}(S)$ by $\gamma_{\hat{\omega},n}(x)$ and $h^{\mathrm{fib}}(\mathcal{S})$, respectively, and ask for the limit to hold for $\mathbf Q$-a.e.~$\tilde{\omega}$.}
\vskip .1cm
(ii) The Central Limit Theorem from \cite[Corollary 2.1]{atilla} for the quantity $m_{T,S}(n,x)$ for two NTFM's $T$ and $S$ asks for $T$ to satisfy the zero-property and for $S$ to satisfy a property called the weak invariance principle, which seems to be quite strong. By asking that $S$ satisfies the zero-property, we have obtained the Central Limit Theorem under the somewhat less restrictive CLT-property on the random number system $\mathcal T$. This implies that the Central Limit Theorem from \cite[Corollary 2.1]{atilla} also holds under the assumptions that the NTFM $T$ has the CLT-property and $S$ has the zero-property.
}\end{remark}

\section{Examples involving well-known number expansions}\label{s:examples}

Below we consider some specific examples of random number systems with relations to number expansions. 

\begin{ex}[Random integer base expansions]\label{x:integers}
Consider a sequence of integers $2 \le N_1 < N_2 < N_3 < \ldots$. Set $I = \{ N_1, N_2, \ldots \}$, let $(p_i)_{i \in I}$ be a strictly positive probability vector and let $\mathbf P$ be the corresponding Bernoulli measure on $I^{\mathbb{N}}$. Assume that
\begin{equation}\label{q:entropydelta}
\quad \sum_{i\in I} p_i \log^2 i < \infty.
\end{equation}
For each $i \in I$, let $T_i (x) = i x \pmod 1$. The maps $T_i$ are non-singular and piecewise monotone and $C^1$ with respect to the partitions $\alpha_i = \{ A_{i,j}\}_{j \ge 0}$ given by
\[ A_{i,j} = \begin{cases}
\Big[ \frac{j}{i}, \frac{j+1}{i} \Big), & \text{if } 0 \le j \le i-1,\\
\emptyset , & \text{otherwise}.
\end{cases}\]
Thus, conditions (r2) and (r4) are satisfied. The countability of $I$ accounts for (r1) and (r3). All maps $T_i$ preserve $\lambda$ and are ergodic with respect to $\lambda$. The invariance of $\mathbf P \times \lambda$ for the skew product $F$ follows from a direct computation and its ergodicity follows from standard results, such as e.g.~\cite[Theorem 5.1]{morita1}. Hence, we get (r6). All maps are expanding, since $DT_i(x) \ge 2$ for all $x \in [0,1)$ and $i \in I$, which implies (r5). Finally, the $\hat{F}$-invariant measure that is obtained by applying Proposition \ref{p:alejan} to $\mathbf{P} \times \lambda$ is $\hat{\mathbf{P}} \times \lambda$, where $\hat{\mathbf{P}}$ is the Bernoulli measure on $I^{\mathbb{Z}}$ associated to $(p_i)_{i \in I}$. Hence, it follows from \eqref{q:entropydelta} that condition \eqref{eq24a} is satisfied, so $\mathcal T = (I, \mathbf P, \{ T_i \}_{i \in I}, \mathbf P\times \lambda, \{ \alpha_i\}_{i \in I})$ is a random number system and thus Theorem~\ref{t:main} applies. Combining Theorem~\ref{t:main}(i), (iii) and \eqref{q:entropydelta} then gives for $\mathbf P$-a.e.~$\omega$ and $\lambda$-a.e.~$x$ that 
\[ \begin{split}
\lim_{n \rightarrow \infty} - \frac{\log \lambda\big(\alpha_{\omega,n}(x)\big)}{n} =\ & h^{\text{fib}} (\mathcal T) =  \int_{I^\mathbb N} \int_{[0,1)} \log |DT_{\omega_1} (x)|\, d\lambda(x) d\mathbf P (\omega)\\
 =\ &  \int_{I^\mathbb N} \log \omega_1 d\mathbf P (\omega) = \sum_{i \in I} p_i \log i < \infty.
\end{split} \]
Note that the right-hand side is the weighted sum of the entropies $h_{\lambda} (T_i) = \log i$.

\vskip .2cm
Note also that the collection $\{ T_i \}_{i \in I}$ does not necessarily fall into the set $\mathcal D$ from the previous section, since there does not need to be a uniform upper bound on the derivatives of the maps $T_i$. We show that $\mathcal T$ satisfies the CLT-property nonetheless. For each $j \in \mathbb{N}$, define the random variable $X_j$ on $I^\mathbb N \times [0,1)$ as
\[ X_j(\omega,x) = -\sum_{A \in \alpha_{\omega_j}} 1_{(T_{\omega}^{j-1})^{-1}A}(x) \log \lambda(A) = -\log \lambda(\alpha_{\omega_j}(T_{\omega}^{j-1}(x))).\]
Then $\{X_j\}_{j\ge 1}$ is an i.i.d.~sequence on $(I^\mathbb N \times [0,1),\mathcal B_I^\mathbb N \times \mathcal B ,\mathbf P \times \lambda)$. Since each map $T_i$ preserves $\lambda$, we obtain
\[ \begin{split}
\mathbb{E}_{\mathbf{P} \times \lambda}(X_j) =\ & -\int_{I^\mathbb N} \int_{[0,1)} \log \lambda(\alpha_{\omega_j}(T_{\omega}^{j-1}(x))) \, d\lambda(x)d\mathbf{P}(\omega)\\
=\ & -\int_{I^\mathbb N} \int_{[0,1)} \log \lambda(\alpha_{\omega_j}(x))\, d\lambda(x)d\mathbf{P}(\omega)\\
=\ & \int_{I^\mathbb N} \log \omega_j \, d\mathbf{P}(\omega) = h^{\text{fib}}(\mathcal T).
\end{split}\]
Similarly,
\begin{equation}\label{q:sigmainteger}
\sigma^2 = \text{Var}(X_j) = \int_{I^{\mathbb{N}}} (\log \omega_j - h^{\text{fib}}(\mathcal T))^2 d\mathbf{P}(\omega) = \sum_{i\in I} p_i \log^2 i - \Big( \sum_{i\in I} p_i \log i \Big)^2.
\end{equation}
It follows from \eqref{q:entropydelta} that $\sigma^2 \in (0, \infty)$. Also $-\log \lambda(\alpha_{\omega,n}(x)) = \sum_{j=1}^n X_j(\omega,x)$, hence from the Central Limit Theorem we get
\[ \frac{-\log \lambda (\alpha_{\omega,n}(x)) - n h^{\text{fib}}(\mathcal T)}{\sigma \sqrt n} = \frac{\sum_{j=1}^n X_j - n \mathbb{E}_{\mathbf{P} \times \lambda}(X_j)}{\sigma \sqrt n} \to \mathcal N(0,1),\]
where the convergence is in distribution with respect to $\mathbf P \times \lambda$. Hence, $\mathcal T$ satisfies the CLT-property with variance $\sigma^2$ and with respect to $\mathbf P \times \lambda$.

\vskip .2cm
Recall that the digit sequence $(d^{\mathcal T}_n(\omega,x))_{n \ge 1}$ was defined in \eqref{q:digitsequence} by setting $d^{\mathcal T}_n(\omega,x)=j_n$ if $T_\omega^{n-1}(x) \in A_{\omega_n,j_n}$. For each $k \ge 1$ we can write
\[ T_\omega^k (x) = \omega_k \big( T_\omega^{k-1} (x) \big) - d^{\mathcal T}_k(\omega,x),\]
so that
\[ x = \frac{d^{\mathcal T}_1(\omega,x)}{\omega_1} + \frac{d^{\mathcal T}_2(\omega,x)}{\omega_1 \omega_2} + \cdots + \frac{d^{\mathcal T}_n(\omega,x)}{\omega_1 \cdots \omega_n} + \frac{T_\omega^n (x)}{\omega_1 \cdots \omega_n}.\]
Since $ \lim_{n \to \infty} \frac{T_\omega^n (x)}{\omega_1 \cdots \omega_n} \le \lim_{n \to \infty} \frac1{2^n} = 0$, we obtain an expansion of $x$ of the form
\[ x = \sum_{n=1}^{\infty} \frac{d^{\mathcal T}_n(\omega,x)}{\prod_{i \in I}i^{c_{n,i}(\omega)}},\]
where $c_{n,i}(\omega) = \#\{ 1\le j \le n \, : \, \omega_j=i\} $. Hence, the random number system $\mathcal T$ produces number expansions of numbers $x \in [0,1)$ in mixed integer bases $N_1, N_2, \ldots$.

\vskip .2cm
Consider the random number system $\mathcal T$ from above and another random number system of this form for $2 \le M_1 < M_2 < \ldots$ with $J= \{ M_1, M_2, \ldots \}$ and a probability vector $(q_j)_{j \in J}$ specifying the Bernoulli measure $\mathbf Q$ on $J^\mathbb N$. Assume that $(q_j)_{j \in J}$ satisfies the condition \eqref{q:entropydelta} as well. Write $ \{ S_j \}_{j \in J}$ for the maps $S_j (x) = j x \pmod 1$ and let $ \{ \gamma_j\}_{j \in J}$ be the corresponding partitions into maximal intervals on which the maps $S_j$ are monotone. For the random number systems $\mathcal T = (I^\mathbb N, \mathbf P, \{T_i\}_{i \in I}, \mathbf P \times \lambda, \{ \alpha_i\}_{i \in I})$ and $\mathcal S =  (J^\mathbb N, \mathbf Q, \{S_j\}_{j \in J}, \mathbf Q \times \lambda, \{\gamma_j \}_{j \in J})$ we obtain
\[ h^{\text{fib}} (\mathcal T) = \sum_{i\in I} p_i \log i, \quad h^{\text{fib}} (\mathcal S) = \sum_{j \in J} q_j \log j.\]
The Random Lochs' Theorem from Theorem~\ref{t:randomlochs} then states that for $\mathbf P \times \mathbf Q$-a.e.~$(\omega, \tilde \omega) \in I^\mathbb N \times J^\mathbb N $ it holds that
\[ \lim_{n \to \infty} \frac{m_{\mathcal T, \mathcal S}(n, \omega, \tilde \omega, x)}{n} = \frac{\sum_{i\in I} p_i \log i}{\sum_{j \in J} q_j \log j} \quad \lambda\text{-a.e}.\]
In other words, given $\omega, \tilde \omega$ and the first $n$ digits $d^{\mathcal T}_1 (\omega,x), \ldots , d^{\mathcal T}_n (\omega,x)$ of an unknown $x$, then typically we can determine approximately the first $n\frac{\sum_{i \in I} p_i \log i}{\sum_{j \in J} q_j \log j}$ digits of $x$ in mixed integer bases $M_1,M_2,\ldots$ generated by the random system $\mathcal S$.

\vskip .2cm
Moreover, if we take the NTFM $S(x) = M x \pmod 1$ for some integer $M \ge 2$, then from Theorem~\ref{t:randomlochs} we obtain for $\mathbf P$-a.e.~$\omega \in I^\mathbb N$,
\[ \lim_{n \to \infty} \frac{m_{\mathcal T, \mathcal S}(n, \omega, \tilde \omega, x)}{n} = \frac{\sum_{i \in I} p_i \log i}{\log M} \qquad \lambda\text{-a.e.}\]
From \cite[Section 3.2]{atilla} we know that $S$ has the zero-property, so that Theorem~\ref{t:clt} gives for each $u \in \mathbb R$ that
\[ \lim_{n \to \infty} \mathbf{P} \times \lambda \Big( \Big\{ (\omega,x) \in I^\mathbb N \times [0,1) \, : \, \frac{m_{\mathcal T, S}(n, \omega,x) - n \frac{\sum_{i\in I} p_i \log i}{\log M}}{\frac{\sigma}{\log M}\sqrt n} \le u \Big\} \Big) = \frac1{\sqrt{2\pi}} \int_{-\infty}^u e^{-t^2/2} \, dt,\]
where $\sigma$ is as in \eqref{q:sigmainteger}.
\end{ex}

\begin{remark}{\rm
Arguments almost identical to the ones in Example~\ref{x:integers} hold for any random system consisting of GLS-transformations. A GLS-transformation, see \cite{BBDK}, is a piecewise linear map $T:[0,1) \to [0,1)$ specified by an at most countable interval partition of $[0,1)$ and for each of these intervals an orientation, such that $T$ maps each interval linearly onto $[0,1)$ with the specified orientation. For example, we can fix some $N \in \{2,3,\ldots\} \cup \{\infty\}$ (the number of branches of the map), $\eta \in (0,1)$ (a lower bound on the slope of the branches) and $q = (q_i)_{i=1}^N \in (0,\eta)^N$ with $\sum_{i=1}^N q_i = 1$ (the sizes of the intervals). The transformation $T_q: [0,1) \rightarrow [0,1)$ given by
\[ T_q(x) = \sum_{i=1}^N  \frac{1}{q_i}\Big(x-\sum_{j=1}^{i-1} q_j \Big) 1_{\big[\sum_{j=1}^{i-1} q_j, \sum_{j=1}^i q_j\big)}(x)\]
is a GLS-transformation mapping each interval $\big[\sum_{j=1}^{i-1} q_j, \sum_{j=1}^i q_j\big)$ linearly and orientation preservingly onto $[0,1)$. We can set $I = \{ q = (q_i)_{i=1}^N \in (0,\eta)^N \, : \, \sum_{i=1}^N q_i = 1\}$, let $\mathbf{P}$ be a $\tau$-invariant probability measure on $I^{\mathbb N}$ and consider the family $\{T_q\}_{q \in I}$. (Note that contrary to in Example~\ref{x:integers} this $I$ is not countable.) In analogy with \eqref{q:entropydelta} we assume
\[ \int_{I^{\mathbb{N}}} \int_{[0,1)} \log^2 \lambda\big(\alpha_{\omega_1}(x)\big) \, d\lambda(x) d\mathbf{P}(\omega) < \infty. \]
Together with the partitions $\alpha_q = \{A_{q,1},A_{q,2},\ldots\}$ with $A_{q,i} = [\sum_{j=1}^{i-1} q_j, \sum_{j=1}^i q_j)$ and measure $\mathbf P \times \lambda$ the system $\mathcal T = (I^\mathbb N, \mathbf P, \{ T_q \}_{q \in I}, \mathbf P \times \lambda, \{ \alpha_q\}_{q \in I})$ is a random number system and satisfies the CLT-property with variance
\[ \sigma^2 = \int_{I^{\mathbb{N}}} \int_{[0,1)} \Big(\log \lambda\big(\alpha_{\omega_1}(x)\big)-h^{\mathrm{fib}}(\mathcal{T})\Big)^2 d\lambda(x) d\mathbf{P}(\omega) \in (0, \infty).\]
Furthermore, in a similar way as in Example~\ref{x:integers} it can be shown that
\[ h^{\mathrm{fib}}(\mathcal{T}) = \int_{I^{\mathbb{N}}} h_{\lambda}(T_{\omega_1})\, d\mathbf{P}(\omega).\]
Number expansions obtained from this system are random versions of what are called {\em generalised L\"uroth series expansions}. A particular instance of this class was studied in \cite{KM}.
}\end{remark}

\begin{figure}[h]
\centering
\subfigure[$T_2$ and $T_3$.]{
\begin{tikzpicture}[scale=2]
\draw[white] (-.3,0)--(1.3,0);
\draw[thick](0,0)node[below]{\small 0}--(.33,0)node[below]{\small $\frac13$}--(.5,0)node[below]{\small $\frac12$}--(.67,0)node[below]{\small $\frac23$}--(1,0)node[below]{\small 1}--(1,1)--(0,1)node[left]{\small 1}--(0,0);
\draw[thick, blue!50] (0,0)--(.5,1)(.5,0)--(1,1);
\draw[thick, green!50!black] (0,0)--(.33,1)(.33,0)--(.67,1)(.67,0)--(1,1);
\draw[dotted](.5,0)--(.5,1)(.33,0)--(.33,1)(.67,0)--(.67,1);
\end{tikzpicture}}
\hspace{1cm}
\subfigure[$T_0$ on the left and $T_1$ on the right.]{
\begin{tikzpicture}[scale=2]
\filldraw[fill=green!50!black, draw=green!70!black] (0,0) rectangle (.09,1);
\draw[dotted](.5,0)--(.5,1)(.33,0)--(.33,1)(.25,0)--(.25,1)(.2,0)--(.2,1);
\draw[thick, green!50!black] (1,0) .. controls (.75,.33) and (.6,.67) .. (.5,1);
\draw[thick, green!50!black] (.5,0) .. controls (.42,.38) and (.36,.78) .. (.33,1);
\draw[thick, green!50!black] (.33,0) .. controls (.29,.45) and (.27,.7) .. (.25,1);
\draw[thick, green!50!black] (.25,0) .. controls (.225,.25) and (.21,.5) .. (.2,1);
\draw[thick, green!50!black] (.2,0) .. controls (.185,.25) and (.175,.5) .. (.17,1);
\draw[thick, green!50!black] (.17,0) .. controls (.155,.25) and (.145,.5) .. (.14,1);
\draw[thick, green!50!black] (.14,0) .. controls (.133,.25) and (.13,.5) .. (.125,1);
\draw[thick, green!50!black] (.125,0) .. controls (.118,.25) and (.115,.5) .. (.11,1);
\draw[thick, green!50!black] (.11,0) .. controls (.105,.25) and (.102,.5) .. (.1,1);
\draw[thick, green!50!black] (.1,0) .. controls (.095,.25) and (.092,.5) .. (.09,1);

\draw[thick](0,0)node[below]{\small 0}--(.2,0)node[below]{\small $\frac15$}--(.25,0)node[below]{\small $\frac14$}--(.33,0)node[below]{\small $\frac13$}--(.5,0)node[below]{\small $\frac12$}--(1,0)node[below]{\small 1}--(1,1)--(0,1)node[left]{\small 1}--(0,0);

\filldraw[fill=blue!50, draw=blue!50] (2.41,0) rectangle (2.5,1);
\draw[thick, blue!50] (1.5,0) .. controls (1.75,.33) and (1.9,.67) .. (2,1);
\draw[thick, blue!50] (2,0) .. controls (2.08,.38) and (2.14,.78) .. (2.17,1);
\draw[thick, blue!50] (2.17,0) .. controls (2.21,.45) and (2.23,.7) .. (2.25,1);
\draw[thick, blue!50] (2.25,0) .. controls (2.275,.25) and (2.29,.5) .. (2.3,1);
\draw[thick, blue!50] (2.3,0) .. controls (2.315,.25) and (2.325,.5) .. (2.33,1);
\draw[thick, blue!50] (2.33,0) .. controls (2.345,.25) and (2.355,.5) .. (2.36,1);
\draw[thick, blue!50] (2.36,0) .. controls (2.367,.25) and (2.37,.5) .. (2.375,1);
\draw[thick, blue!50] (2.375,0) .. controls (2.377,.25) and (2.38,.5) .. (2.39,1);
\draw[thick, blue!50] (2.39,0) .. controls (2.395,.25) and (2.398,.5) .. (2.4,1);
\draw[thick, blue!50] (2.4,0) .. controls (2.405,.25) and (2.408,.5) .. (2.41,1);

\draw[dotted](2,0)--(2,1)(2.17,1)--(2.17,0)(2.25,1)--(2.25,0);

\draw[thick](1.5,0)node[below]{\small 0}--(2,0)node[below]{\small $\frac12$}--(2.17,0)node[below]{\small $\frac23$}--(2.25,0)node[below]{\small $\frac34$}--(2.5,0)node[below]{\small 1}--(2.5,1)--(1.5,1)node[left]{\small 1}--(1.5,0);
\end{tikzpicture}}
\caption{In (a) we see the graphs of the maps $T_2$ (blue) and $T_3$ (green) from Examples~\ref{x:integers}. The maps $T_0$ (Gauss) and $T_1$ (R\'enyi) from Example~\ref{x:gr} are shown in (b).}
\label{f:integerbeta}
\end{figure}
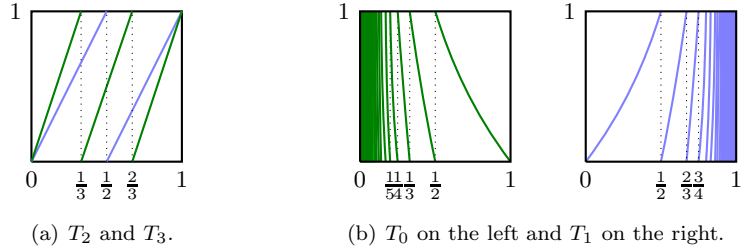

\begin{ex}[Random continued fraction expansions]\label{x:gr}
The Gauss-R\'enyi map is defined by setting $I=\{0,1\}$ with $T_0:[0,1) \to [0,1)$ given by $T_0(0)=0$ and $T_0 (x) = \frac1x \pmod 1$ for $x \neq 0$ (the Gauss map) and $T_1:[0,1) \to [0,1)$ given by $T_1(x) = T_0 (1-x)$ (the R\'enyi map). The graphs are shown in Figure~\ref{f:integerbeta}(b). Fix some $0 < p < 1$ and let $\mathbf P$ be the $(p, 1-p)$-Bernoulli measure on $I^\mathbb N$. The partitions $\alpha_0 = \{ A_{0,j}\}_{j \ge 0}$ and $\alpha_1 = \{ A_{1,j}\}_{j \ge 0}$ are given by
\[ A_{0,j} = \Big( \frac1{j+2}, \frac1{j+1} \Big] \quad \text{ and } \quad A_{1,j} = \Big[ \frac{j}{j+1}, \frac{j+1}{j+2} \Big), \quad j \ge 0.\]
In \cite{KKV17} it was proven that there exists a measure $\rho_p$ equivalent to $\lambda$ such that $\mathbf P \times \rho_p$ is invariant and ergodic with respect to the skew product $F$ with a density $\frac{d\rho_p}{d\lambda}$ that is bounded away from zero and is of bounded variation, so in particular bounded. For the system $\mathcal T = (I, \mathbf P, \{ T_i\}_{i =0,1}, \mathbf P \times \rho_p, \{ \alpha_i \}_{i=0,1})$ conditions (r1)--(r6) follow straightforwardly. For (r7), note that for each $j \ge 0$ we have
\[ H_{\mathbf P \times \rho_p} (\Delta) = \sum_{j \ge 0} p \rho_p \Big( \Big( \frac1{j+2}, \frac1{j+1} \Big] \Big) + (1-p) \rho_p \Big( \Big[ \frac{j}{j+1}, \frac{j+1}{j+2} \Big) \Big).\]
Since the density of $\rho_p$ is bounded, it follows that $H_{\mathbf P \times \rho_p} (\Delta) < \infty$. Thus (r7) holds and $\mathcal T = (I, \mathbf P, \{ T_i\}_{i =0,1}, \mathbf P \times \rho_p, \{ \alpha_i \}_{i=0,1})$ is a random number system. For the fiber entropy we obtain from Theorem~\ref{t:main}(iii) that
\[ h^{\text{fib}} (\mathcal T) = \int_{I^\mathbb N} \int_{[0,1)} \log |DT_{\omega_1} (x)| \, d\rho_p (x) d\mathbf P(\omega)= -2 \int_0^1 p\log x + (1-p)\log(1-x) \, d\rho_p (x).\]
It follows from equation (8) in \cite{KKV17} that the digit sequences $( d^\mathcal T_k (\omega,x) )_{k \ge 1}$ from this random number system give the semi-regular continued fraction expansions of real numbers $x \in [0,1]$ by
\[ x = \cfrac1{\omega_1 + d^\mathcal T_1(\omega,x) + \cfrac{(-1)^{\omega_1}}{\omega_2 + d^\mathcal T_2(\omega,x)+  \cfrac{(-1)^{\omega_2}}{\omega_3 + d^\mathcal T_3(\omega,x) + \ddots}}}.\]
\end{ex}

\begin{ex}(Random $\beta$-expansions in alternate base)\label{ex2.3}
Fix two constants $1 < \eta < \delta$ and let $J = [\eta, \delta]$. For each $\beta \in J$, let $S_\beta:[0,1) \to [0,1), \, x \mapsto \beta x \pmod 1$ be the $\beta$-transformation, which is the piecewise linear map with slope $\beta$ on the partition $\gamma_\beta = \{ C_{\beta,j}\}_{j \ge 0}$ given by
\[ C_{\beta,j} = \begin{cases}
\big[ \frac{j}{\beta}, \frac{j+1}{\beta} \big), & \text{if } 0 \le j < \lceil \beta \rceil-1\\
\big[ \frac{\lceil \beta \rceil -1}{\beta},1 \big), & \text{if } j= \lceil \beta \rceil-1,\\
\emptyset, & \text{otherwise},
\end{cases}\]
where $\lceil \beta \rceil$ indicates the smallest integer not smaller than $\beta$. See Figure~\ref{f:betas}(a) for some graphs. The study of $\beta$-transformations was initiated by R\'enyi in \cite{renyi}. The $\beta$-transformations are related to $\beta$-expansions of real numbers, which are expressions of the form
\[ x = \sum_{k \ge 1} \frac{b_k}{\beta^k}, \quad b_k \in \{0,1, \ldots, \lceil \beta \rceil -1\}.\]
Fix some periodic sequence
\[ u = (u_1, u_2, \ldots, u_m, u_1, u_2, \ldots, u_m, u_1, \ldots ) \in J^\mathbb N\] 
of period length $m \ge 2$ and consider the shift invariant measure $\mathbf Q$ on $J^\mathbb N$ defined by
\begin{equation}\label{q:Q}
\mathbf Q = \frac1m \sum_{k=0}^{m-1} \delta_{\tau^k u},
\end{equation}
where $\delta_y$ denotes the Dirac measure at the point $y \in J^{\mathbb{N}}$. We define the map $\psi: \{1,2,\ldots,m\} \times [0,1) \rightarrow J^{\mathbb{N}} \times [0,1)$ by
\[ \psi(i,x) = \big((u_i,u_{i+1},\ldots,u_m,u_1,u_2,\ldots,u_m,u_1,\ldots),x),\]
which is measurable if we put on $\{1,2,\ldots,m\} \times [0,1)$ the $\sigma$-algebra
\[ \mathcal{A} = \Big\{ \bigcup_{i=1}^m \{i\} \times B_i \, :\,  i \in \{1,\ldots,m\}, B_i \in \mathcal B\Big\}.\]
Defining the transformation $T_u: \{1,2,\ldots,m\} \times [0,1) \rightarrow \{1,2,\ldots,m\} \times [0,1)$ by
\[ T_u(i,x) = \big((i+1) \bmod m, T_{u_i}(x)\big),\]
it follows from \cite{CCD} that there are measures $\mu_{u,1},\ldots,\mu_{u,m}$ on $([0,1),\mathcal{B})$ such that the measure $\mu_u$ on $(\{1,\ldots,m\} \times [0,1),\mathcal{A})$ given by
\[ \mu_u\Big(\bigcup_{i=1}^m \{i\} \times B_i \Big) = \frac{1}{m} \sum_{i=1}^m \mu_{u,i}(B_i), \qquad B_1,\ldots, B_m \in \mathcal{B},\]
is an ergodic invariant measure for $T_u$ that is equivalent to the measure $\lambda_u$ on $(\{1,\ldots,m\} \times [0,1),\mathcal{A})$ given by
\[ \lambda_u\Big(\bigcup_{i=1}^m \{i\} \times B_i \Big)  = \frac{1}{m} \sum_{i=1}^m \lambda(B_i), \qquad B_1,\ldots, B_m \in \mathcal{B}.\]
Since $\lambda_u = (\mathbf{Q} \times \lambda) \circ \psi^{-1}$ and $\psi$ is an isomorphism between the dynamical systems $(\{1,\ldots,m\} \times [0,1),\mathcal{A},\mu_u,T_u)$ and $(J^{\mathbb{N}} \times [0,1),\mathcal{B}_J^{\mathbb{N}} \times \mathcal{B},\mathbf{Q} \times \rho,F)$ with $\rho = \sum_{i=1}^m \mu_{u,i}$ and $F$ the skew product associated to $\{ S_\beta\}_{\beta \in J}$, it follows that $\mathbf Q \times \rho$ is an ergodic invariant measure for $F$ and that $\rho$ is equivalent to $\lambda$. Since $\delta$ is an upper bound for $J$, the collection $\Delta$ is finite and thus $H_{\mathbf Q \times \rho} (\Delta) < \infty$. Figure~\ref{f:betas}(b) illustrates the sets $\Delta(j)$ for $\eta \in (1,2)$ and $\delta \in (2,3)$. It follows that $\mathcal S = (J, \mathbf Q, \{ S_\beta \}_{\beta \in J}, \mathbf Q \times \rho, \{ \gamma_\beta \}_{\beta \in J})$ is a random number system. From Theorem~\ref{t:main}(iii) we get
\[ h^{\text{fib}} (\mathcal S) = \int_{J^\mathbb N} \int_{[0,1)} \log \omega_1 \, d\rho(x) d\mathbf Q(\omega) = \frac1m \sum_{i=1}^m \log u_i.\]

\vskip .2cm
By the definition of the digit sequence $(d^{\mathcal S}_n(\omega,x))_{n \ge 1}$ we can write for each $\omega \in J^{\mathbb{N}}$, $x \in [0,1)$ and $n \ge 1$ that
\[ S_\omega^n  (x) = \omega_n  S_\omega^{n-1} (x)  - d^{\mathcal S}_n(\omega,x),\]
so that
\[ x = \frac{d^{\mathcal S}_1(\omega,x)}{\omega_1} + \frac{d^{\mathcal S}_2(\omega,x)}{\omega_1 \omega_2} + \cdots + \frac{d^{\mathcal S}_n(\omega,x)}{\omega_1 \cdots \omega_n} + \frac{S_\omega^n (x)}{\omega_1 \cdots \omega_n}.\]
Since $\lim_{n \to \infty} \frac{S_\omega^n (x)}{\omega_1 \cdots \omega_n} \leq \lim_{n \to \infty} \frac1{\eta^n} = 0$, for each $x \in [0,1)$ and $\omega \in J^{\mathbb{N}}$ we obtain the {\em random mixed $\beta$-expansion}
\[ x = \sum_{n\ge 1} \frac{d^{\mathcal S}_n(\omega,x)}{\omega_1 \cdots \omega_n}.\]
With our choice of $\mathbf Q$ from \eqref{q:Q} it holds that for $\mathbf Q$-a.e.~$\omega \in J^\mathbb N$ with $\omega_1 = u_1$ the random mixed $\beta$-expansions produced by the system $\mathcal S$ are the {\em greedy $(u_1, \ldots, u_m)$-expansions in alternate base} that are the object of study in \cite{CCD}.

\vskip .2cm
With Theorem~\ref{t:randomlochs} we can compare the semi-regular continued fraction digits from the random continued fraction map from Example~\ref{x:gr} with the alternate base greedy $\beta$-expansions. If we let $\mathcal T = (I, \mathbf P, \{ T_i\}_{i =0,1}, \mathbf P_p \times \rho_p, \{ \alpha_i \}_{i=0,1})$ be the system from Example~\ref{x:gr} and let $(J, \mathbf Q, \{ S_\beta \}_{\beta \in J}, \mathbf Q \times \rho, \{ \gamma_\beta \}_{\beta \in J})$ be the system from Example~\ref{ex2.3}, then Theorem~\ref{t:randomlochs} tells us that for $\mathbf P_p \times \mathbf Q$-a.e.~$(\omega, \tilde \omega) \in I^\mathbb N \times J^\mathbb N$,
\[ \lim_{n \to \infty} \frac{m_{\mathcal T, \mathcal S}(n, \omega, \tilde \omega,x)}{n} = \frac{- 2 \int_{[0,1)} p\log x + (1-p)\log (1-x) \, d\rho_p (x)}{\frac1m \sum_{i=1}^m \log u_i} \quad \lambda\text{-a.e.}\]
\end{ex}

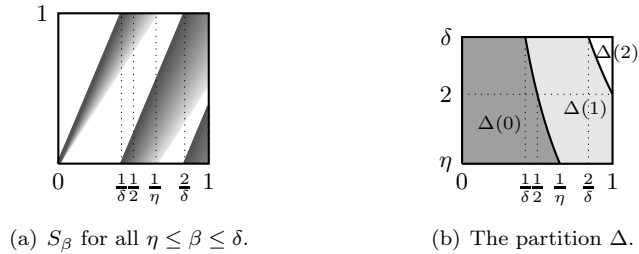
\begin{figure}[h]
\centering
\subfigure[$S_\beta$ for all $ \eta \le \beta \le \delta$.]{
\begin{tikzpicture}[scale=2]
\draw[white] (-.5,0)--(1.5,0);
\draw[line width=0.55mm,gray!20](0,0)--(.65,1);
\draw[line width=0.42mm,gray!25](0,0)--(.64,1);
\draw[line width=0.42mm,gray!30](0,0)--(.63,1);
\draw[line width=0.42mm,gray!35](0,0)--(.62,1);
\draw[line width=0.42mm,gray!40](0,0)--(.61,1);
\draw[line width=0.42mm,gray!45](0,0)--(.6,1);
\draw[line width=0.42mm,gray!50](0,0)--(.59,1);
\draw[line width=0.42mm,gray!55](0,0)--(.58,1);
\draw[line width=0.42mm,gray!60](0,0)--(.57,1);
\draw[line width=0.42mm,gray!65](0,0)--(.56,1);
\draw[line width=0.42mm,gray!70](0,0)--(.55,1);
\draw[line width=0.42mm,gray!75](0,0)--(.54,1);
\draw[line width=0.42mm,gray!80](0,0)--(.53,1);
\draw[line width=0.42mm,gray!85](0,0)--(.52,1);
\draw[line width=0.42mm,gray!90](0,0)--(.51,1);
\draw[line width=0.42mm,gray!95](0,0)--(.5,1);
\draw[line width=0.42mm,gray](0,0)--(.49,1);
\draw[line width=0.42mm,gray!95!black](0,0)--(.48,1);
\draw[line width=0.42mm,gray!90!black](0,0)--(.47,1);
\draw[line width=0.42mm,gray!85!black](0,0)--(.46,1);
\draw[line width=0.42mm,gray!80!black](0,0)--(.45,1);
\draw[line width=0.42mm,gray!75!black](0,0)--(.44,1);
\draw[line width=0.42mm,gray!70!black](0,0)--(.43,1);
\draw[line width=0.42mm,gray!65!black](0,0)--(.42,1);

\draw[line width=0.42mm,gray!20](.65,0)--(1,.538);
\draw[line width=0.42mm,gray!25](.64,0)--(1,.563);
\draw[line width=0.42mm,gray!30](.63,0)--(1,.587);
\draw[line width=0.42mm,gray!35](.62,0)--(1,.613);
\draw[line width=0.42mm,gray!40](.61,0)--(1,.639);
\draw[line width=0.42mm,gray!45](0.6,0)--(1,.667);
\draw[line width=0.42mm,gray!50](0.59,0)--(1,.695);
\draw[line width=0.42mm,gray!55](0.58,0)--(1,.724);
\draw[line width=0.42mm,gray!60](0.57,0)--(1,.754);
\draw[line width=0.42mm,gray!65](0.56,0)--(1,.786);
\draw[line width=0.42mm,gray!70](0.55,0)--(1,.818);
\draw[line width=0.42mm,gray!75](0.54,0)--(1,.852);
\draw[line width=0.42mm,gray!80](0.53,0)--(1,.887);
\draw[line width=0.42mm,gray!85](0.52,0)--(1,.923);
\draw[line width=0.42mm,gray!90](0.51,0)--(1,.961);
\draw[line width=0.42mm,gray!95](0.5,0)--(1,1);
\draw[line width=0.42mm,gray](0.49,0)--(.98,1);
\draw[line width=0.42mm,gray!95!black](0.48,0)--(.96,1);
\draw[line width=0.42mm,gray!90!black](0.47,0)--(.94,1);
\draw[line width=0.42mm,gray!85!black](0.46,0)--(.92,1);
\draw[line width=0.42mm,gray!80!black](0.45,0)--(.9,1);
\draw[line width=0.42mm,gray!75!black](0.44,0)--(.88,1);
\draw[line width=0.42mm,gray!70!black](0.43,0)--(.86,1);
\draw[line width=0.42mm,gray!65!black](0.42,0)--(.84,1);

\draw[line width=0.45mm,gray](0.98,0)--(1,.041);
\draw[line width=0.45mm,gray!95!black](0.96,0)--(1,.083);
\draw[line width=0.45mm,gray!90!black](0.94,0)--(1,.128);
\draw[line width=0.45mm,gray!85!black](0.92,0)--(1,.174);
\draw[line width=0.45mm,gray!80!black](0.9,0)--(1,.222);
\draw[line width=0.45mm,gray!75!black](0.88,0)--(1,.273);
\draw[line width=0.45mm,gray!70!black](0.86,0)--(1,.326);
\draw[line width=0.45mm,gray!65!black](0.84,0)--(1,.381);

\draw[dotted] (.42,0)--(.42,1)(.65,0)--(.65,1)(.5,0)--(.5,1)(.84,0)--(.84,1);

\draw[thick](0,0)node[below]{\small 0}--(.42,0)node[below]{\small $\frac1{\delta}$}--(.5,0)node[below]{\small $\frac12$}--(.64,0)node[below]{\small $\frac1{\eta}$}--(.84,0)node[below]{\small $\frac2{\delta}$}--(1,0)node[below]{\small 1}--(1,1)--(0,1)node[left]{\small 1}--(0,0);
\end{tikzpicture}}
\hspace{1cm}
\subfigure[The partition $\Delta$.]{
\begin{tikzpicture}[scale=2]
\draw[white] (-.5,2)--(1.5,2);
\filldraw[fill=gray!75, draw=gray!75] (.65,1.538)--(0,1.538)--(0,2.381)--(.42,2.381)--(.5,2)--(.65,1.528);
\filldraw[fill=gray!20, draw=gray!20] (.65,1.538)--(1,1.538)--(1,2)--(.84,2.381)--(.42,2.381)--(.5,2)--(.65,1.528);

\draw[thick](0,1.538)node[below]{\small 0}--(.42,1.538)node[below]{\small $\frac1{\delta}$}--(.5,1.538)node[below]{\small $\frac12$}--(.65,1.538)node[below]{\small $\frac1{\eta}$}--(.84,1.538)node[below]{\small $\frac2{\delta}$}--(1,1.538)node[below]{\small 1}--(1,2.381)--(0,2.381)node[left]{\small $\delta$}--(0,2)node[left]{\small $2$}--(0,1.538)node[left]{\small $\eta$};

\draw[dotted](0,2)--(1,2);
\draw[smooth,thick,samples =20, domain=.42:.65] plot(\x,{1/ \x});
\draw[smooth,thick,samples =20, domain=.84:1] plot(\x,{2/ \x});

\node at (.25,1.8) {\scriptsize $\Delta(0)$};
\node at (.82,1.9) {\scriptsize $\Delta(1)$};
\node at (1.03,2.27) {\scriptsize $\Delta(2)$};

\draw[dotted] (.42,1.538)--(.42,2.381);
\draw[dotted] (.84,1.538)--(.84,2.381);
\draw[dotted] (.5,1.538)--(.5,2);

\end{tikzpicture}}
\caption{In (a) we see the graphs of all the maps $S_\beta$ for $\beta \in [\eta,\delta]$ from Example~\ref{ex2.3} for some $1 < \eta < 2 < \delta < 3$. Each shade of grey corresponds to one graph. In (b) we see the elements of the partition $\Delta$ for values $\eta$ and $\delta$ as in (a).}
\label{f:betas}
\end{figure}

\bibliographystyle{alpha} 
\bibliography{Random_Lochs_Theorem} 
\end{document}